\documentclass[10pt]{article}
\usepackage{float} 
\usepackage{amsmath, amssymb, amsthm}
\newtheorem{theorem}{Theorem}
\newtheorem{hypothesis}[theorem]{Hypothesis}
\newtheorem{lemma}[theorem]{Lemma}
\newtheorem{Example}[theorem]{Example}
\newtheorem{remark}[theorem]{Remark}
\newtheorem{corollary}[theorem]{Corollary}

\def\R{\mathbb R}
\def\P{\mathbb P}
\def\Q{{\mathbb Q}}

\def\1{{ \mathbbm{1}}}
\def\E{\mathbb{E}}
\def\N{\mathbb{N}}
\newcommand{\di}{\displaystyle}
\usepackage{xcolor}
\usepackage{graphicx}
\usepackage{tikz}
\def\1{{ \mathbbm{1}}}
\usepackage[T1]{fontenc}
\usepackage{amsmath,amssymb,amstext}
\usepackage{mathrsfs}
\usepackage{bbm}

\usepackage[a4paper,margin=1.25in]{geometry}
\usepackage{orcidlink}

\usepackage{xcolor}
\definecolor{darkgreen}{rgb}{0.0,0.65,0.0}
\definecolor{BLUE}{RGB}{0,0,255}
\definecolor{GREEN}{rgb}{0.0,0.55,0.0}
\definecolor{pinegreen}{rgb}{0.0, 0.47, 0.44}
\definecolor{non-photoblue}{rgb}{0.64, 0.87, 0.93}
\definecolor{magicmint}{rgb}{0.67, 0.94, 0.82}
\definecolor{laserlemon}{rgb}{1.0, 1.0, 0.13}
\definecolor{darkolivegreen}{rgb}{0.33, 0.42, 0.18}
\definecolor{darkgoldenrod}{rgb}{0.72, 0.53, 0.04}
\definecolor{cof}{RGB}{219,144,71}
\definecolor{pur}{RGB}{186,146,162}
\definecolor{greeo}{RGB}{91,173,69}
\definecolor{greet}{RGB}{52,111,72}
\definecolor{brightube}{rgb}{0.82, 0.62, 0.91}
\definecolor{darkcyan}{rgb}{0.0, 0.55, 0.55}
\usepackage{natbib}
\usepackage{hyperref}
\hypersetup{
  colorlinks=true,
  linkcolor=darkgreen,     % Table of contents section titles, equations, etc.
  citecolor=darkgreen,       % citations like (Beck et al., 2019)
  urlcolor=blue         % URLs
}

\begin{document}

\title{Nonzero-Sum Stochastic Differential Games for Controlled Convection-Diffusion SPDEs
%Game Controlled stochastic partial differential equations  in Spatially Heterogeneous Systems
}
\author{Nacira Agram\orcidlink{0000-0003-1662-0215}\footnotemark[1]\; %\; Bernt Øksendal\footnotemark[2] \, 
and\, Eya Zougar\orcidlink{0009-0002-7278-5567}\footnotemark[3]}

\footnotetext[1]{Department of Mathematics, KTH Royal Institute of Technology 100 44, Stockholm, Sweden. Digital Futures, Stockholm. Email: nacira@kth.se.}

%\footnotetext[2]{Department of Mathematics, University of Oslo, Oslo, Norway 
%Email: oksendal@math.uio.no}

\footnotetext[2]{Univ. Polytechnique Hauts-de-France, INSA Hauts-de-France, CERAMATHS - Laboratoire de Mat\'eriaux C\'eramiques et de Math\'ematiques, F-59313 Valenciennes, France. 
Email:  Eya.Zougar@uphf.fr}
\date{\today}
\maketitle

\begin{abstract}
This paper studies a two-player nonzero-sum stochastic differential game governed by a controlled convection–diffusion stochastic partial differential equation (SPDE) with spatially heterogeneous coefficients. The diffusion and transport operators depend on the players’ controls, allowing each agent to influence the system dynamics. 
We prove the existence and uniqueness of solutions to both the forward uncontrolled SPDE and the associated adjoint backward SPDE (BSPDE) in a Hilbert space framework. Using a Hamiltonian approach, we derive sufficient and necessary maximum principles characterizing Nash equilibria. Special attention is given to operators with piecewise constant coefficients, where interface transmission conditions arise naturally.
As an illustration, we provide two examples from composite materials where the game structure models the interaction between different material phases in a diffusion process.
\end{abstract}


\section{Introduction}

Stochastic partial differential equations (SPDEs) have emerged as powerful tools for modeling systems influenced by both spatial variability and random fluctuations in time. These equations find applications in a wide range of disciplines, including physics \cite{Lejay2004}, engineering, biology \cite{Nicaise1984}, ecology \cite{CantrellCosner1999} and finance \cite{AgramOksendal2021}, where uncertainty and spatial structure are intrinsic to the dynamics. Typical examples include stochastic with discontinuous or piecewise constant coefficients, see for example \cite{ZZ4}, as well as SPDEs involving mixed or fractional operators \cite{Zougar2024,Zougar2026}, arising in the modeling of diffusion and propagation in composite or spatially heterogeneous media.

In many practical situations, the evolution of such systems is not only influenced by random noise but also subject to external control or intervention. This leads naturally to the study of controlled SPDEs, where decision-makers seek to optimize certain performance criteria by influencing the system through time dependent control variables.
An even richer framework arises when multiple agents, possibly with conflicting objectives, interact with the system. This brings us into the realm of stochastic differential games governed by SPDEs. In this setting, each player applies control strategies to influence the evolution of the system, aiming to optimize their own cost functional. Such game-theoretic formulations are particularly relevant in competitive physical systems, economics, and optimal design problems in materials science.

This paper investigates a two-player stochastic differential game in which the state dynamics are governed by a controlled SPDE with heterogeneous coefficients. The control acts through both the diffusion and drift operators, allowing each player to influence the transport and diffusion mechanisms in a distinct manner.

To describe the interaction between the players’ control strategies and the system dynamics, we consider the following controlled Convection-Diffusion SPDE:
\begin{equation}\label{eqcontrol.01}
\left\{
\begin{array}{rcl}
 dY(t,x)&=& \mathcal{A}^{u_1,u_2}_{\rho,a,b}Y(t,x) \, dt + \kappa(t,x,Y(t,x), u_1(t),u_2(t)) \, dt\\ 
 &+& \sigma(t,x,Y(t,x), u_1(t),u_2(t)) \, dB(t); \qquad (t,x) \in [0,T] \times \mathbb{R}, \\
Y(0,x) &=& \xi(x), \quad x \in \mathbb{R},
\end{array}
\right.
\end{equation}
where \( Y(t, x) \) represents the state of the material at time \( t \) and spatial position \( x \), such as temperature or stress.  And, \( B(t) \) is a one-dimensional Brownian motion with zero mean and covariance
\[
\mathbb{E}[B_t B_s] = t \land s.
\]
The functions $\kappa$, $\sigma$, and the initial condition $\xi$ will be defined in Section~2, where the corresponding structural and regularity assumptions are stated.

The operator \( \mathcal{A}^{u_1,u_2}_{\rho,a,b} \) in the equation (\ref{eqcontrol.01})  is defined as

\begin{equation}\label{operatorcontrol0}
\mathcal{A}^{u_1,u_2}_{\rho,a,b}(x)= \frac{\rho(x)}{2} \frac{d}{dx} \left(  a(x,u_1) \frac{d}{dx} \right) + b(x,u_2) \frac{d}{dx},
\end{equation}
appears to model a differential operator that describes a system with spatially varying properties, where the terms $u_i,\, i=1,2$ are control functions that influence the behavior of the system. 
In the general case, the coefficients $a(x, u_1)$, $b(x, u_2)$, and $\rho(x)$ are taken as measurable and sufficiently regular functions that may vary continuously or discontinuously  with respect to the spatial variable $x$, at several locations within the domain, and depend on control variables $u_1$ and $u_2$. Specifically, $a(x, u_1)$ models the local diffusion properties influenced by the control $u_1$, $b(x, u_2)$ represents a controlled transport or drift effect, and $\rho(x)$ reflects spatially varying material density or mass distribution. 
This formulation is flexible enough to incorporate a wide range of realistic phenomena, including materials with non-uniform structures, spatial asymmetries, and regions where the physical properties shift due to external inputs or design strategies.

In the setting of a two-player control game, the control functions $u_i(t),$ for $i=1,2$, represent the actions of each player and may correspond to external factors such as:
\begin{itemize} 
\item Uniform temperature settings applied during curing or heating processes.
\item Global mechanical loads exerted on the material.
\item  Homogeneous processing parameters like pressure or chemical treatments.
 \item Other relevant influences depending on the specific application.
\end{itemize}
Each player aims to optimize their own cost functional, which balances deviations from desired system metrics (e.g., temperature distribution or stress profile) and the cost of applying their control $u_i(t), \,i=1,2$. By strategically selecting their controls, players influence the system evolution cooperatively or competitively, shaping the material’s final properties under uncertainty.

A particular case of the SPDE considered in this work is when the coefficients are piecewise constant, reflecting abrupt changes in material properties across different spatial regions. Indeed,
the functions \( a(x, u_1) \), $b(x,u_2)$  and \( \rho(x) \)  are piecewise functions representing the properties of the composite material. Specifically:
\begin{equation}\label{coffab}
a(x, u_1) = a^-(u_1)  \mathbbm{1}_{\{ x \leq 0 \}} + a^+(u_1)  \mathbbm{1}_{\{ x > 0 \}}, \quad 
b(x, u_2) = b^-(u_2)  \mathbbm{1}_{\{ x \leq 0 \}} + b^+(u_2)  \mathbbm{1}_{\{ x > 0 \}}
\end{equation}
\begin{equation}\label{coffrho}
    \text{and}\quad  \rho(x) = \rho^-  \mathbbm{1}_{\{ x \leq 0 \}} + \rho^+  \mathbbm{1}_{\{ x > 0 \}}.
\end{equation}
Here, $a^{\pm}, b^{\pm} \in C^1(\mathbb{R})$ for $i = 1, 2$, and $\rho^{\pm}$ are real constants. The spatial dependence of $a(x, u_1)$ reflects the heterogeneous nature of composite materials, while the control input $u_1(t)$, which does not vary with $x$, models uniform external influences such as temperature settings or mechanical loading applied across the entire material. In contrast, the spatial variation in $b(x, u_2)$ captures asymmetric transport phenomena, allowing the modeling of external influences that vary across the domain. This enables the description of localized control actions that impact the direction and magnitude of transport or stress differently in distinct regions of the medium. And, the coefficient $\rho$ plays the role of a spatial weight or density factor in the diffusion operator. \\

\begin{figure}[ht!]
\begin{center}
    \begin{tikzpicture}[thick,scale=2.5]
    \coordinate (A1) at (0, 0);
    \coordinate (A2) at (0, 1);
    \coordinate (A3) at (1, 1);
    \coordinate (A4) at (1, 0);
    \coordinate (B1) at (0.3, 0.3);
    \coordinate (B2) at (0.3, 1.3);
    \coordinate (B3) at (1.3, 1.3);
    \coordinate (B4) at (1.3, 0.3);
\draw[very thick] (A1) -- (A2);
    \draw[very thick] (A2) -- (A3);
    \draw[very thick] (A3) -- (A4);
    \draw[very thick] (A4) -- (A1);
 \draw[dashed] (A1) -- (B1);
    \draw[dashed] (B1) -- (B2);
    \draw[very thick] (A2) -- (B2);
    \draw[very thick] (B2) -- (B3);
    \draw[very thick] (A3) -- (B3);
    \draw[very thick] (A4) -- (B4);
    \draw[very thick] (B4) -- (B3);
    \draw[dashed] (B1) -- (B4);
\draw[fill=greeo,opacity=0.6] (A1) -- (B1) -- (B4) -- (A4);
    \draw[fill=greeo,opacity=0.5] (A1) -- (A2) -- (A3) -- (A4);
    \draw[fill=greeo,opacity=0.6] (A1) -- (A2) -- (B2) -- (B1);
    \draw[fill=greeo,opacity=0.6] (B1) -- (B2) -- (B3) -- (B4);
    \draw[fill=greeo,opacity=0.6] (A3) -- (B3) -- (B4) -- (A4);
    \draw[fill=greeo,opacity=0.6] (A2) -- (B2) -- (B3) -- (A3);

    \coordinate (A1) at (1, 0);
    \coordinate (A2) at (1, 1);
    \coordinate (A3) at (2, 1);
    \coordinate (A4) at (2, 0);
    \coordinate (B1) at (1.3, 0.3);
    \coordinate (B2) at (1.3, 1.3);
    \coordinate (B3) at (2.3, 1.3);
    \coordinate (B4) at (2.3, 0.3);

    \draw[very thick] (A1) -- (A2);
    \draw[very thick] (A2) -- (A3);
    \draw[very thick] (A3) -- (A4);
    \draw[very thick] (A4) -- (A1);

    \draw[dashed] (A1) -- (B1);
    \draw[dashed] (B1) -- (B2);
    \draw[very thick] (A2) -- (B2);
    \draw[very thick] (B2) -- (B3);
    \draw[very thick] (A3) -- (B3);
    \draw[very thick] (A4) -- (B4);
    \draw[very thick] (B4) -- (B3);
    \draw[dashed] (B1) -- (B4);

    \draw[fill=pur,opacity=0.6] (A1) -- (B1) -- (B4) -- (A4);
    \draw[fill=pur,opacity=0.5] (A1) -- (A2) -- (A3) -- (A4);
    \draw[fill=black!70,opacity=0.6] (A1) -- (A2) -- (B2) -- (B1);
    \draw[fill=pur,opacity=0.6] (B1) -- (B2) -- (B3) -- (B4);
    \draw[fill=pur,opacity=0.6] (A3) -- (B3) -- (B4) -- (A4);
    \draw[fill=pur,opacity=0.6] (A2) -- (B2) -- (B3) -- (A3);

\draw[ultra thick,->] (-0.25,-0.25) -- ++(2.7,0) node[right] {$x:$  Spatial Domain };
\draw[ thick] (1, -0.16) node[below] {\color{cyan}$\bullet$} ;
\draw[red] (1.2, -0.1) node[left] {\color{cyan} $x=0$};
\draw[red] (0.3, -0.1) node[left] {\color{cyan} Region 1};
\draw[red] (2.3, -0.1) node[left] {\color{cyan} Region 2};
\draw[ thick] (0.65,0.8) node[below] {\bf Phase  1};
\draw[ thick] (0.65,0.65) node[below] { $\rho^-$};

%Controls
\draw[cyan,ultra thick,<-] (1,1.3)-- (1.3,1.5) node[align=center] {\color{cyan} Controls Game $(u_1(t),u_2(t))$\\};
\draw[cyan,ultra thick,<-] (1.8,1.3)-- (1.3,1.5);

%\draw[ thick] (0.65,0.5) node[below] { ${a^-(u_1)}$}; 
\draw[ thick] (-0.6,0.8) node[below] { ${ \color{greeo}a(x,u_1)=a^-(u_1)}$};
\draw[ thick] (-0.6,0.5) node[below] { ${ \color{greeo}b(x,u_2)=b^-(u_2)}$};
\draw[ thick] (1.65,0.8) node[below] {\bf Phase  2};
\draw[ thick] (1.65,0.65) node[below] { $\rho^+$};
\draw[ thick] (2.9,0.8) node[below] { ${ \color{purple}a(x,u_1)=a^+(u_1)}$};
\draw[ thick] (2.9,0.5) node[below] { ${ \color{purple}b(x,u_2)=b^+(u_2)}$};

\draw[gray,ultra thick,<-] (2.2,1.1)-- (2.7,1.2) node[right] {Desired metric  };
\draw[gray] (2.7,1) node[right] { (e.g,  Heat, stress, etc ) };
%\draw[ultra thick, ->] (0,2) ++ (-50:.55) arc (-50:250:.55 and .25);
\end{tikzpicture}
\caption{Illustration of a two-player control game in a composite system with piecewise constant coefficients}
\end{center}
\end{figure}
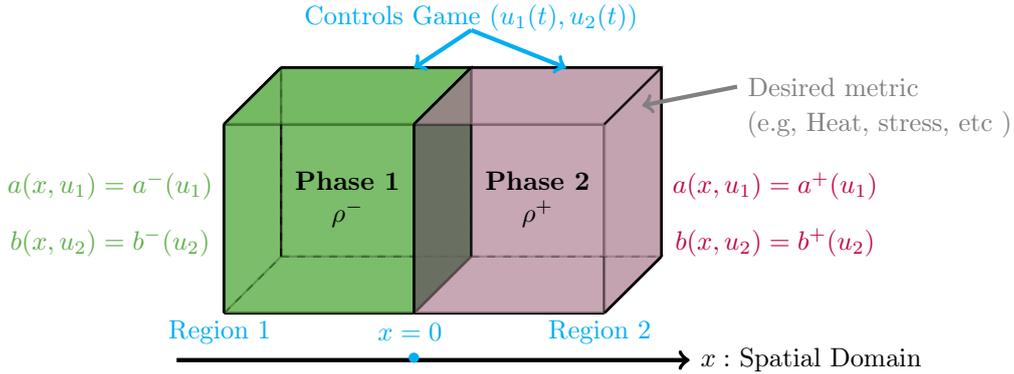

A particular case of the controlled SPDE model, where the drift term $b \equiv 0$ and $a,\rho$ are defined respectively in (\ref{coffab}) and (\ref{coffrho}), was introduced in \cite{ATZ}. This simplification focuses on pure diffusion dynamics without convection or transport effects, allowing for the analysis of systems dominated by stochastic diffusion phenomena.

The aim of this paper is to investigate a two-player nonzero-sum stochastic differential game in which each player seeks to enhance the performance of a spatially heterogeneous material. Specifically, each agent minimizes an individual cost functional that penalizes deviations from prescribed target states (such as a desired temperature profile) together with the cost of the applied control. Because the state dynamics depend on both control actions, the problem gives rise to a strategic interaction between the players through a common stochastic evolution.

Our analysis is based on a Hamiltonian framework. We derive the associated backward stochastic partial differential equations (BSPDEs) and characterize optimal strategies in terms of maximum principles. To illustrate the theoretical results, we provide an application motivated by composite material design, where the two players represent distinct phases or mechanisms interacting to influence the effective properties of the medium.

The paper is organized as follows. In Section 2, we introduce the functional framework and state the main assumptions on the coefficients, and we establish existence and uniqueness results for the uncontrolled forward SPDE. Section 3 is devoted to the formulation of the two-player nonzero-sum stochastic differential game, where we derive the associated Hamiltonian system and the adjoint backward SPDE. We then establish sufficient and necessary maximum principles characterizing Nash equilibria. Then, we present explicit representations in the constant coefficient case and discuss the impact of piecewise constant heterogeneous coefficients, highlighting the corresponding transmission conditions. Finally, in Section 4, we illustrate the theoretical results with some applications to heat regulation in composite materials.

\section{Stochastic Convection--Diffusion SPDE}\label{Sec2}
Let $T>0$ and let $(\Omega,\mathcal{F},\mathbb{P})$ be a probability space with a filtration 
$\mathbb{F}=(\mathcal{F}_t)_{t\in[0,T]}$ satisfying the usual conditions. 
Let $B=(B(t))_{t\in[0,T]}$ be a one-dimensional $\mathbb{F}$-Brownian motion. 
\subsection{Problem Formulation}

Consider the following stochastic partial differential equations (SPDEs) used to model composite materials with spatially varying characteristics. Throughout this work, the model will be referred to as the
\emph{general stochastic convection--diffusion equation} and is given by 
\begin{equation}\label{eq.1}
\left\{
\begin{array}{rcl}
 dY(t,x)&=& \mathcal{A}_{\rho,a,b}Y(t,x) \, dt + \kappa(t,x,Y(t,x)) \, dt\\ 
 &+& \sigma(t,x,Y(t,x)) \, dB(t); \quad (t,x) \in [0,T] \times \mathbb{R}, \\
Y(0,x) &=& \xi(x), \quad x \in \mathbb{R},
\end{array}
\right.
\end{equation}
where \( Y(t, x) \) denotes the state variable of the material at time $t$ and spatial position $x$, representing physical quantities such as temperature, stress, or concentration, depending on the context of the application.
The operator \( \mathcal{A}_{\rho,a,b}\) is a linear second-order parabolic operator in $L^2(\R)$, defined by:
\begin{equation}\label{operator}
\mathcal{A}_{\rho,a,b}= \frac{\rho(x)}{2} \frac{d}{dx} \left(  a(x) \frac{d}{dx} \right) + b(x) \frac{d}{dx}.
\end{equation}
with $a,\rho$ and $b$ possibly discontinuous. The coefficients $\rho,a,b$, corresponds to what is commonly referred to as a \emph{constant skew diffusion}. 
Such coefficients  may be discontinuous at different
points and/or may have several points of discontinuities. \medskip

The operator $\mathcal A_{\rho,a,b}$ considered in this work provides a
general framework that extends several diffusion models previously studied
in the literature. In particular, operators with piecewise constant
coefficients and interface effects were introduced and analyzed in the
context of skew diffusions and composite media in \cite{lejay1,lejay2,ATZ}.
Related stochastic and control-oriented models with discontinuous spatial
coefficients were further investigated in our earlier works, for example \cite{ZZ1,ZZ2,ZZ5,ZZ4} and \cite{TudorZougar2026}. The present formulation allows for
a more general class of spatially heterogeneous coefficients, possibly with
multiple discontinuities, and is therefore well suited for the stochastic
control and game-theoretic analysis developed in the next sections.
 
 From a physical viewpoint, the operator $\mathcal{A}_{\rho,a,b}$ models a {\emph{diffusion-advection process}}, describing how a quantity, such as heat, particle concentration, or population density, spreads within a medium while also being transported by an external flow. The first term represents a diffusion effect, where 
$a(x)$ acts as a spatially dependent diffusivity coefficient, modulated by the density function 
$\rho(x)$. The second term,  represents advection, where 
$b(x)$ functions as a velocity field that moves the quantity in a particular direction. Such an operator appears in various applications, including heat conduction in moving fluids, wave propagation in heterogeneous media, quantum transport in semiconductors, and ecological models describing species migration. Its structure makes it particularly relevant for studying systems where both diffusive spreading and directed transport are present. 
\paragraph{Hypothesis.}Throughout the paper, we work under the following standing assumptions on the
spatial coefficients $(\rho,a,b)$ and on the nonlinear functions
$\kappa$, $\sigma$, and $\xi$ involved in the model.

\begin{hypothesis}[Coefficients $(\rho,a,b)$ of the operator]\label{hypo1}
Let $l_1<l_2$ be extended real numbers in $\overline{\mathbb{R}}$, and let
$\rho,a,b:[l_1,l_2]\to\mathbb{R}$ be measurable functions.
We assume that there exist some positive constants $0<\lambda\le\Lambda$ such that for any $x\in[l_1,l_2]$
\begin{align*}
&\lambda \le \rho(x) \le \Lambda, \\
\lambda \le a(x) \le &\Lambda,\quad \text{and}\quad 
|b(x)| \le \Lambda.
\end{align*}
%$\text{for a.e. } x\in[l_1,l_2]$
\end{hypothesis}

\begin{hypothesis}[Assumptions on $(\kappa,\sigma)$]\label{hypo2}
The coefficients 
\[
\kappa,\sigma : \Omega \times [0,T] \times \mathbb{R} \times \mathbb{R} \to \mathbb{R}
\]
are measurable and Lipschitz continuous with respect to the state variable.
More precisely, there exists a constant $L>0$ such that, for all
$(\omega,t,x)\in \Omega \times [0,T] \times \mathbb{R}$ and all
$y,y'\in\mathbb{R}$,
\[
|\kappa(t,x,y)-\kappa(t,x,y')|
\vee
|\sigma(t,x,y)-\sigma(t,x,y')|
\le L |y-y'|,
\]
and
\[
|\kappa(t,x,y)| \vee |\sigma(t,x,y)|
\le L (1+|y|).
\]
\end{hypothesis}

\begin{hypothesis}\label{hypo3}
 The function $\xi:\R\mapsto\R$ is a nonrandom, measurable, and bounded function. That is, there exists a positive constant $M,$ such that $\left|\xi(x)\right|\leq M,\, \forall x\in\R.$
\end{hypothesis}

\paragraph{Notation.}Throughout the paper, we work on the real line $\mathbb{R}$.
We  introduce the Hilbert space
\begin{equation}\label{spaceH}
\mathbb{H} := L^2(\mathbb R,\rho(x)^{-1}\,dx).
\end{equation}

For $\varphi,\psi \in C_c^\infty(\mathbb{R})$, we define the weighted inner product on $\mathbb{H}$ by
\[
\langle \varphi,\psi\rangle_{\mathbb{H}}
=
\int_{\mathbb{R}} \varphi(x)\,\psi(x)\,\frac{dx}{\rho(x)},
\]
which extends by density to $\mathbb{H}$, and induces the norm
\[
\|\psi\|_{\mathbb{H}}
=
\left( \int_{\mathbb{R}} |\psi(x)|^2 \frac{dx}{\rho(x)} \right)^{1/2}.
\]

\subsection{Some properties of the operator %${\cal A}_{\rho,a,b}$
} 

We now briefly recall some analytical properties of the operator $\mathcal{A}_{\rho,a,b}$ that will be useful in the sequel, distinguishing between the smooth-coefficient case and the heterogeneous case with
piecewise constant coefficients.

\begin{lemma}[Linearity]
\label{lem:linearity}
Let $\rho,a,b:\mathbb R\to\mathbb R$ be measurable functions satisfying
Hypothesis~\ref{hypo1}. Then, for any $\varphi_i\in C_c^\infty(\mathbb R)$,
$i=1,2$, and any $\alpha_i\in\mathbb R$, $i=1,2$, the operator
$\mathcal A_{\rho,a,b}$ satisfies
\[
\mathcal{A}_{\rho,a,b}\left(\sum_{i=1}^2\alpha_i\varphi_i\right)
=\sum_{i=1}^2\alpha_i\,\mathcal{A}_{\rho,a,b}(\varphi_i)
\]
\end{lemma}

\begin{proof}
The result follows from the linearity of differentiation and multiplication by
fixed functions. Indeed, for $\varphi_1,\varphi_2\in C_c^\infty(\mathbb R)$ and
$\alpha_1,\alpha_2\in\mathbb R$, we have
\[
\begin{aligned}
\mathcal A_{\rho,a,b}\!\left(\alpha_1\varphi_1+\alpha_2\varphi_2\right)
&=
\frac{\rho(x)}{2}\frac{d}{dx}\!\left(
a(x)\frac{d}{dx}(\alpha_1\varphi_1+\alpha_2\varphi_2)
\right)
+
b(x)\frac{d}{dx}(\alpha_1\varphi_1+\alpha_2\varphi_2)
\\
&=
\alpha_1\left[
\frac{\rho(x)}{2}\frac{d}{dx}\!\left(a(x)\frac{d\varphi_1}{dx}\right)
+
b(x)\frac{d\varphi_1}{dx}
\right]
+
\alpha_2\left[
\frac{\rho(x)}{2}\frac{d}{dx}\!\left(a(x)\frac{d\varphi_2}{dx}\right)
+
b(x)\frac{d\varphi_2}{dx}
\right],
\end{aligned}
\]
which yields the desired identity.
\end{proof}

\begin{lemma}[Coercivity in $\mathbb{H}$]
\label{lem:coercivity}

Assume that the coefficients $(\rho,a,b)$ satisfy Hypothesis~\ref{hypo1}.
Let
\begin{equation}\label{space def}
V := H^1(\mathbb R),    
\end{equation}
endowed with the norm
\[
\|u\|_V^2 := \|u'\|_{L^2(\mathbb R)}^2 + \|u\|_{\mathbb{H}}^2.
\]
Then the operator $-\mathcal A_{\rho,a,b}:V\to V^*$ satisfies the coercivity
condition that is, there exist some positive constants $\alpha>0$ and
$\lambda_0\ge0$, depending only on the bounds in Hypothesis~\ref{hypo1}, such that
\[
2\langle -\mathcal A_{\rho,a,b}u,u\rangle_{V^*,V}
+ \lambda_0 \|u\|_{\mathbb{H}}^2
\;\ge\;
\alpha \|u\|_V^2,
\qquad \forall u\in V.
\]
\end{lemma}

\begin{proof}
The operator $\mathcal A_{\rho,a,b}$ is understood in the variational sense:
for all $u,v\in V$,
\[
\langle \mathcal A_{\rho,a,b}u,v\rangle_{V^*,V}
:=
-\frac12\int_{\mathbb R}\rho(x)a(x)u'(x)v'(x)\,dx
+\int_{\mathbb R} b(x)u'(x)v(x)\,dx.
\]

By Hypothesis~\ref{hypo1}, the coefficients satisfy
\[
\lambda \le \rho(x)\le \Lambda,
\qquad
\lambda \le a(x)\le \Lambda,
\qquad
|b(x)|\le \Lambda,
\quad \text{a.e. }x\in\mathbb R.
\]
In particular,
\[
\rho(x)a(x)\ge \lambda^2 \quad \text{a.e. }x\in\mathbb R.
\]

Taking $v=u\in V$, we obtain
\[
\langle \mathcal A_{\rho,a,b}u,u\rangle
=
-\frac12\int_{\mathbb R}\rho(x)a(x)|u'(x)|^2\,dx
+\int_{\mathbb R} b(x)u'(x)u(x)\,dx.
\]

The diffusion term is estimated using uniform ellipticity:
\[
-\frac12\int_{\mathbb R}\rho a\,|u'|^2\,dx
\le
-\frac{\lambda^2}{2}\|u'\|_{L^2(\mathbb R)}^2.
\]

For the drift term, by Cauchy--Schwarz and Young’s inequality, for any
$\varepsilon>0$,
\[
\left|\int_{\mathbb R} b\,u'u\,dx\right|
\le
\Lambda\|u'\|_{L^2}\|u\|_{L^2}
\le
\varepsilon\|u'\|_{L^2}^2
+
\frac{\Lambda^2}{4\varepsilon}\|u\|_{L^2}^2.
\]

Since $\rho$ is bounded above and below, the norms
$\|\cdot\|_{L^2(\mathbb R)}$ and $\|\cdot\|_{\mathbb{H}}$ are equivalent. Hence there
exists a constant $C_\rho>0$ such that
\[
\|u\|_{L^2(\mathbb R)}^2 \le C_\rho \|u\|_{\mathbb{H}}^2.
\]

Combining the above estimates yields
\[
\langle \mathcal A_{\rho,a,b}u,u\rangle
\le
-\Big(\frac{\lambda^2}{2}-\varepsilon\Big)\|u'\|_{L^2}^2
+ C_\varepsilon \|u\|_{\mathbb{H}}^2,
\]
for some constant $C_\varepsilon>0$.

Multiplying by $-2$ and choosing $\varepsilon=\lambda^2/4$, we obtain
\[
-2\langle \mathcal A_{\rho,a,b}u,u\rangle
\ge
\frac{\lambda^2}{2}\|u'\|_{L^2}^2
- C \|u\|_{\mathbb{H}}^2,
\]
for some constant $C>0$ depending only on $\lambda$ and $\Lambda$.

Choosing $\lambda_0>C+1$, it follows that
\[
-2\langle \mathcal A_{\rho,a,b}u,u\rangle
+ \lambda_0\|u\|_{\mathbb{H}}^2
\ge
\alpha\big(\|u'\|_{L^2}^2+\|u\|_{\mathbb{H}}^2\big)
=
\alpha\|u\|_V^2,
\]
with $\alpha:=\min\{\lambda^2/2,1\}>0$. This proves the coercivity condition.
\end{proof}

\paragraph{Smooth coefficient case.}
\begin{lemma}
\label{rem:basic-props}
Let $(\rho,a,b)$ satisfy Hypothesis~\ref{hypo1}.
If $a\in C^1(\mathbb{R})$, then for any $\varphi\in C_c^\infty(\mathbb{R})$,
the operator $\mathcal{A}_{\rho,a,b}$ can be written in the following non-divergence form
\[
(\mathcal{A}_{\rho,a,b}\varphi)(x)
=\frac{\rho(x)a(x)}{2}\,\varphi''(x)
+\left(\frac{\rho(x)a'(x)}{2}+b(x)\right)\varphi'(x),
\qquad x\in\R.
\]
\end{lemma}
\begin{proof}
Let $\varphi\in C_c^\infty(\mathbb{R})$. Since $a\in C^1(\mathbb{R})$, the product
rule yields
\[
\frac{d}{dx}\!\left(a(x)\frac{d\varphi}{dx}(x)\right)
=
a(x)\varphi''(x)+a'(x)\varphi'(x).
\]
Substituting this identity into the definition of
$\mathcal{A}_{\rho,a,b}$, we obtain
\[
(\mathcal{A}_{\rho,a,b}\varphi)(x)
=
\frac{\rho(x)}{2}
\left(a(x)\varphi''(x)+a'(x)\varphi'(x)\right)
+
b(x)\varphi'(x),
\]
which simplifies to the claimed non-divergence form.
\end{proof}

\begin{lemma}[Adjoint operator in the smooth case]\label{lem:adjoint-smooth}
Assume that the coefficients $a$, $\rho$, and $b$ belong to $C^1(\mathbb{R})$.
Then the adjoint of the operator $\mathcal{A}_{\rho,a,b}$ in $\mathbb{H}$
%$L^2(\mathbb{R},\rho(x)^{-1}dx)$ 
is given by
\[
(\mathcal{A}^*_{\rho,a,b}\psi)(x)
=
\frac{\rho(x)}{2}\frac{d}{dx}\!\left(a(x)\frac{d\psi}{dx}(x)\right)
-\rho(x)\frac{d}{dx}\!\left(\frac{b(x)}{\rho(x)}\psi(x)\right),
\qquad x\in\mathbb{R},
\]
for any $\psi\in C_c^\infty(\mathbb{R})$.
\end{lemma}

\begin{proof}
Let $\varphi,\psi\in C_c^\infty(\R).$ %and define the weighted inner product of $H$.
%\begin{equation}\label{prod-scalar}
%\langle \varphi,\psi\rangle_{L^2(\rho^{-1})}
%:=\int_{\mathbb{R}} \varphi(x)\,\psi(x)\,\frac{dx}{\rho(x)}.
%\end{equation}
Assuming that
$a,\rho,b\in C^1(\mathbb{R})$, and using the non-divergence form given in Remark~\ref{rem:basic-props}, we have for any $\psi,\varphi\in C_c^\infty(\mathbb{R})$,
\begin{align*}
\langle \mathcal{A}_{\rho,a,b}\psi,\varphi\rangle_{H}
&=\int_{\mathbb{R}}
\left(
\frac{\rho(x)a(x)}{2}\psi''(x)
+\left(\frac{\rho(x)a'(x)}{2}+b(x)\right)\psi'(x)
\right)\varphi(x)\,\frac{dx}{\rho(x)}\\
&=\frac12\int_{\mathbb{R}} a(x)\psi''(x)\varphi(x)\,dx
+\int_{\mathbb{R}}\psi'(x)\left(\frac{a'(x)}{2}
+\frac{b(x)}{\rho(x)}\right)\varphi(x)\,dx\\
&=: I_1+I_2.
\end{align*}

Since $\varphi$ and $\psi$ have compact support, all boundary terms vanish.

For $I_1$, an integration by parts yields
\[
I_1
=\frac12\int_{\mathbb{R}} (a\psi')'\,\varphi\,dx
=-\frac12\int_{\mathbb{R}} a(x)\psi'(x)\varphi'(x)\,dx
=\frac12\int_{\mathbb{R}} (a\varphi')'(x)\,\psi(x)\,dx.
\]
Therefore,
\[
I_1
=\int_{\mathbb{R}}
\psi(x)\,\frac{\rho(x)}{2}(a\varphi')'(x)\,\frac{dx}{\rho(x)}
=\left\langle \psi,\frac{\rho}{2}(a\varphi')'\right\rangle_{\mathbb{H}}.
\]

For $I_2$, another integration by parts gives
\[
I_2
=\int_{\mathbb{R}} \psi'(x)\,\frac{b(x)}{\rho(x)}\,\varphi(x)\,dx
=-\int_{\mathbb{R}} \psi(x)\,\frac{d}{dx}\!\left(\frac{b(x)}{\rho(x)}\varphi(x)\right)\,dx.
\]
Rewriting this in the weighted inner product form,
\[
I_2
=\int_{\mathbb{R}}
\psi(x)\left(-\rho(x)\frac{d}{dx}\!\left(\frac{b(x)}{\rho(x)}\varphi(x)\right)\right)\frac{dx}{\rho(x)}
=\left\langle \psi,-\rho\,\frac{d}{dx}\!\left(\frac{b}{\rho}\varphi\right)\right\rangle_{\mathbb{H}}.
\]

Combining $I_1$ and $I_2$, we obtain
\[
\langle \mathcal{A}_{\rho,a,b}\psi,\varphi\rangle_{\mathbb{H}}
=
\left\langle \psi,
\frac{\rho}{2}(a\varphi')'
-\rho\,\frac{d}{dx}\!\left(\frac{b}{\rho}\varphi\right)
\right\rangle_{\mathbb{H}}.
\]
Hence the adjoint operator $\mathcal{A}^*_{\rho,a,b}$ in $\mathbb{H}$ is
\[
(\mathcal{A}^*_{\rho,a,b}\varphi)(x)
=
\frac{\rho(x)}{2}\frac{d}{dx}\!\left(a(x)\frac{d\varphi}{dx}(x)\right)
-\rho(x)\frac{d}{dx}\!\left(\frac{b(x)}{\rho(x)}\varphi(x)\right),
\qquad x\in\mathbb{R},
\]
as claimed.
\end{proof}

\begin{remark}\label{Remark_adjoint}
\begin{enumerate}
\item If $b\equiv 0$, the adjoint operator $\mathcal{A}_{\rho,a,0}^*$ coincides with
$\mathcal{A}_{\rho,a,0}$; hence $\mathcal{A}_{\rho,a,0}$ is self-adjoint in
$\mathbb{H}$.

\item In the constant coefficient case $a(x)\equiv a_0$, $\rho(x)\equiv\rho_0$ and
$b(x)\equiv b_0$, the adjoint operator reduces to
\[
\mathcal A^*_{\rho,a,b}\psi
=
\frac{\rho_0 a_0}{2}\psi''-b_0\psi'.
\]
Hence $\mathcal A^*_{\rho,a,b}=\mathcal A_{\rho,a,-b}$, and the operator is
self-adjoint if and only if $b_0=0$.
\end{enumerate}
\end{remark}

%\paragraph{Non-Smooth coefficient case.}
\paragraph{Piecewise constant coefficient case.}
Assume that the coefficients $(\rho,a,b)$ are piecewise constant functions  of
the form
\begin{equation}\label{coeff-piecewise}
\begin{aligned}
a(x) &= a^-\,\mathbbm{1}_{(-\infty,0]}(x)+a^+\,\mathbbm{1}_{(0,+\infty)}(x),\\
\rho(x) &= \rho^-\,\mathbbm{1}_{(-\infty,0]}(x)+\rho^+\,\mathbbm{1}_{(0,+\infty)}(x),\\
b(x) &= b^-\,\mathbbm{1}_{(-\infty,0]}(x)+b^+\,\mathbbm{1}_{(0,+\infty)}(x),
\end{aligned}
\end{equation}
where $a^\pm,\rho^\pm>0$ and $b^\pm\in\mathbb{R}$ are constants.

\begin{lemma}[Adjoint operator in the piecewise constant case]\label{lem:adjoint-piecewise}
Under the above piecewise-constant assumption on $(\rho,a,b)$, for any $\phi,\varphi\in C_c^\infty(\mathbb{R})$, the adjoint of
$\mathcal{A}_{\rho,a,b}$ in $\mathbb{H}$ satisfies
\[
\begin{aligned}
\langle \phi, \mathcal{A}_{\rho,a,b}^*\varphi\rangle_{\mathbb{H}}
&=
\left\langle \phi,
C_\pm^{\rho,a}\,\varphi'(0)\,\delta_0
+
C_\pm^{\rho,b}\,\varphi(0)\,\delta_0
\right\rangle
+
\langle \phi, \mathcal{B}_{\rho,a,b}^*\varphi\rangle_{\mathbb{H}},
\end{aligned}
\]
where
\[
(\mathcal{B}_{\rho,a,b}^*\varphi)(x)
=
\frac{\rho(x)}{2}\frac{d}{dx}\!\left(a(x)\frac{d\varphi}{dx}(x)\right)
-
b(x)\frac{d\varphi}{dx}(x),
\qquad x\neq 0,
\]
and the interface coefficients are given by
\[
C_\pm^{\rho,a}=\frac{\rho^-}{2}(a^+-a^-),
\qquad
C_\pm^{\rho,b}
=
-\left(\frac{\rho^-}{\rho^+}b^+ + b^-\right).
\]
\end{lemma}

\begin{proof}
Let $\varphi,\phi \in C_c^\infty(\mathbb{R}).$
From (\ref{operator}), we decompose
\[
\langle \mathcal A_{\rho,a,b}\phi,\varphi\rangle_{\mathbb{H}}
= J_a + J_b,
\]
where
\[
J_a := \frac12\int_{\mathbb{R}} \frac{d}{dx}\!\left(a(x)\phi'(x)\right)\varphi(x)\,dx,
\qquad
J_b := \int_{\mathbb{R}} \phi'(x)\frac{b(x)}{\rho(x)}\varphi(x)\,dx.
\]

In what follows, we treat separately the terms $J_a$ and $J_b$.

\textbf{Treatment of $J_a$.}
Using the definition of the derivative in the sense of distributions applied to
the locally integrable function
$x \mapsto a(x)\phi'(x)$, we obtain
\begin{align*}
J_a
&= -\frac12\int_{\mathbb{R}} a(x)\phi'(x)\varphi'(x)\,dx \\
&= -\frac12\int_{\mathbb{R}^*} a(x)\phi'(x)\varphi'(x)\,dx \\
&= -\frac{a^+}{2}\int_{(0,+\infty)} \phi'(x)\varphi'(x)\,dx
   -\frac{a^-}{2}\int_{(-\infty,0)} \phi'(x)\varphi'(x)\,dx.
\end{align*}

Integrating by parts on each half-line yields
\[
\int_{(0,+\infty)} \phi'(x)\varphi'(x)\,dx
=
-\phi(0)\varphi'(0)
-\int_{(0,+\infty)} \phi(x)\varphi''(x)\,dx,
\]
and
\[
\int_{(-\infty,0)} \phi'(x)\varphi'(x)\,dx
=
\phi(0)\varphi'(0)
-\int_{(-\infty,0)} \phi(x)\varphi''(x)\,dx.
\]

Therefore,
\[
\begin{aligned}
J_a
&= \frac12(a^+-a^-)\phi(0)\varphi'(0)  +\int_{\mathbb{R}^*} \phi(x)\,
\frac{\rho(x)}{2}\frac{d}{dx}\!\left(a(x)\varphi'(x)\right)
\frac{dx}{\rho(x)} \\
&= \big\langle \phi,
C_\pm^{\rho,a}\,\varphi'(0)\delta_0 \big\rangle
+ \langle \phi,\mathcal B^1_{\rho,a}\varphi\rangle_{\mathbb{H}},
\end{aligned}
\]
where
\[
C_\pm^{\rho,a}=\frac{\rho^-}{2}(a^+-a^-),
\qquad
\mathcal B^1_{\rho,a}\varphi(x)
=\frac{\rho(x)}{2}\frac{d}{dx}\!\left(a(x)\varphi'(x)\right).
\]

\textbf{Treatment of $J_b$.} We write
\[
\begin{aligned}
J_b
&= \int_{\mathbb{R}^*} \phi'(x)\frac{b(x)}{\rho(x)}\varphi(x)\,dx \\
&= \frac{b^+}{\rho^+}\int_{(0,+\infty)} \phi'(x)\varphi(x)\,dx
 +\frac{b^-}{\rho^-}\int_{(-\infty,0)} \phi'(x)\varphi(x)\,dx.
\end{aligned}
\]

Integrating by parts on each half-line gives
\[
\int_{(0,+\infty)} \phi'(x)\varphi(x)\,dx
=
-\phi(0)\varphi(0)
-\int_{(0,+\infty)} \phi(x)\varphi'(x)\,dx,
\]
and
\[
\int_{(-\infty,0)} \phi'(x)\varphi(x)\,dx
=
-\phi(0)\varphi(0)
-\int_{(-\infty,0)} \phi(x)\varphi'(x)\,dx.
\]

Hence,
\[
\begin{aligned}
J_b
&= -\left(\frac{b^+}{\rho^+}+\frac{b^-}{\rho^-}\right)\varphi(0)\phi(0) 
 -\int_{\mathbb{R}^*} b(x)\varphi'(x)\phi(x)\frac{dx}{\rho(x)} \\
&= \big\langle \phi,
C_\pm^{\rho,b}\,\varphi(0)\delta_0 \big\rangle
- \langle \phi,\mathcal B^2_{\rho,b}\varphi\rangle_{\mathbb{H}},
\end{aligned}
\]
with
\[
C_\pm^{\rho,b}
=-\left(\frac{\rho^-}{\rho^+}b^+ + b^-\right),
\qquad
\mathcal B^2_{\rho,b}\varphi(x)=b(x)\varphi'(x).
\]
Combining the expressions of $J_a$ and $J_b$ yields the claimed formula for the adjoint operator.
\end{proof}

\begin{remark}[Case $b=0$]
Assume that $b\equiv 0$.
Then the adjoint operator $\mathcal A_{\rho,a,0}^*$ coincides with the formal
adjoint of the second-order operator given in \cite{ATZ}, in the
sense of distributions,
\[
\mathcal A_{\rho,a,0}^*\varphi
=
\mathcal A_{\rho,a,0}\varphi
+
C_\pm^{\rho,a}\,\varphi'(0)\,\delta_0.
\]
In particular, the singular contribution at $x=0$ arises solely from the jump of
the diffusion coefficient $a$.
\end{remark}
\subsection{Forward SPDE existence and unicity}

\begin{theorem}[Existence and Uniqueness of a Mild Forward Solution]\label{exitforward}
Assume that Hypotheses \textnormal{\ref{hypo1}}, 
\textnormal{\ref{hypo2}} and \textnormal{\ref{hypo3}} hold.
Let $(P_t)_{t\ge0}$ be the $C_0$-semigroup on $\mathbb{H}.$ %generated by the
%realization of the operator $\mathcal A_{\rho,a,b}$, defined by
%\[(P_t \varphi)(x)=\int_{\mathbb{R}} p(t,x,y)\,\varphi(y)\,\frac{dy}{\rho(y)},
%\]
%where $p(t,x,y)$ denotes the associated transition kernel.
Then the forward SPDE given in \eqref{eq.1} admits a unique mild solution
\[
Y \in L^2\big(\Omega; C([0,T];\mathbb{H})\big),
\]
which satisfies, for all $t\in[0,T]$,
\begin{equation}\label{mild_forward}
Y(t)=P_t\xi+\int_0^t P_{t-s}\kappa\big(s,\cdot,Y(s,\cdot)\big)\,ds+
\int_0^t P_{t-s}\sigma\big(s,\cdot,Y(s,\cdot)\big)\,dB(s),
\end{equation}
where the first integral is understood in the classical sense in $\mathbb{H}$ and the second one in the It\^o sense in $\mathbb{H}$.
Moreover, the solution is pathwise unique.
\end{theorem}

\begin{proof}
Let $(P_t)_{t\ge0}$ be the $C_0$-semigroup associated with the operator
$\mathcal A_{\rho,a,b}$, given by the kernel representation
\[
(P_t\varphi)(x)=\int_{\R} p(t,x,y)\,\varphi(y)\,\frac{dy}{\rho(y)},
\]
where $p(t,x,y)$ denotes the associated transition kernel.
By Proposition~2 (Aronson-type estimate) in \cite{lejay1}, there exist
constants $C_1,C_2>0$ such that for all $t\in(0,T]$ and $x,y\in\R$,
\begin{equation}\label{eq:aronson}
p(t,x,y)\le \frac{C_1}{\sqrt{t}}\exp\!\left(-\frac{C_2|x-y|^2}{t}\right).
\end{equation}

In particular, the kernel estimate implies that $(P_t)$ is bounded on $\mathbb{H}$:
there exists $M_T>0$ such that
\begin{equation}\label{eq:Ptbounded}
\|P_t\|_{\mathcal L(\mathbb{H})}\le M_T,\qquad t\in[0,T],
\end{equation}
and thus the deterministic and stochastic convolutions below are well-defined
as $\mathbb{H}$-valued integrals.

Define the maps $F,G:[0,T]\times \mathbb{H}\to \mathbb{H}$ by
\[
(F(t,u))(x)=\kappa(t,x,u(x)),\qquad (G(t,u))(x)=\sigma(t,x,u(x)).
\]
A mild solution of \eqref{eq.1} is an $\mathbb{H}$-valued predictable process $Y$ such that
for all $t\in[0,T]$,
\begin{equation}\label{eq:mildagain}
Y(t)=P_t\xi+\int_0^t P_{t-s}F(s,Y(s))\,ds+\int_0^t P_{t-s}G(s,Y(s))\,dB(s).
\end{equation}
The proof relies on a fixed-point argument based on Picard iteration.
\paragraph{Existence.}
Set $Y^{0}(t):=P_t\xi$ and define recursively, for $n\ge0$,
\begin{equation}\label{eq:picard}
Y^{n+1}(t):=P_t\xi+\int_0^t P_{t-s}F(s,Y^{n}(s))\,ds
+\int_0^t P_{t-s}G(s,Y^{n}(s))\,dB(s),\qquad t\in[0,T].
\end{equation}
Using \eqref{eq:Ptbounded}, Cauchy--Schwarz, BDG inequality and the growth bounds
from Hypotheses \ref{hypo2}--\ref{hypo3}, one obtains for each $n$,
\[
\E\sup_{t\le T}\|Y^{n}(t)\|_{\mathbb{H}}^2<\infty,
\]
hence $Y^{n}\in L^2(\Omega;C([0,T];\mathbb{H}))$ and all integrals in \eqref{eq:picard}
are meaningful.

Fix $T_0\in(0,T]$ (to be chosen) and set
\[
\Delta^{n}(t):=Y^{n+1}(t)-Y^{n}(t),\qquad t\in[0,T_0].
\]
Subtracting \eqref{eq:picard} at ranks $n$ and $n-1$ yields
\[
\Delta^{n}(t)=\int_0^t P_{t-s}\Big(F(s,Y^{n}(s))-F(s,Y^{n-1}(s))\Big)\,ds
+\int_0^t P_{t-s}\Big(G(s,Y^{n}(s))-G(s,Y^{n-1}(s))\Big)\,dB(s).
\]
Using $(a+b)^2\le 2a^2+2b^2$, \eqref{eq:Ptbounded}, Cauchy--Schwarz and BDG, we get
\begin{align*}
\E\sup_{t\le T_0}\|\Delta^{n}(t)\|_{\mathbb{H}}^2
&\le
2 M_T^2\,T_0 \,\E\int_0^{T_0}\|F(s,Y^{n}(s))-F(s,Y^{n-1}(s))\|_{\mathbb{H}}^2\,ds \\
&\quad
+2 C_{\mathrm{BDG}} M_T^2\,\E\int_0^{T_0}\|G(s,Y^{n}(s))-G(s,Y^{n-1}(s))\|_{\mathbb{H}}^2\,ds.
\end{align*}
Applying the Lipschitz bounds from Hypotheses \ref{hypo2}--\ref{hypo3},
there exist constants $L_F,L_G$ such that
\[
\|F(s,u)-F(s,v)\|_{\mathbb{H}}\le L_F\|u-v\|_{\mathbb{H}},\qquad
\|G(s,u)-G(s,v)\|_{\mathbb{H}}\le L_G\|u-v\|_{\mathbb{H}}.
\]
Hence
\begin{align*}
\E\sup_{t\le T_0}\|\Delta^{n}(t)\|_{\mathbb{H}}^2
&\le
2M_T^2\Big(T_0^2L_F^2+C_{\mathrm{BDG}}T_0L_G^2\Big)\,
\E\sup_{t\le T_0}\|\Delta^{n-1}(t)\|_{\mathbb{H}}^2.
\end{align*}
Set
\[
q(T_0):=2M_T^2\Big(T_0^2L_F^2+C_{\mathrm{BDG}}T_0L_G^2\Big).
\]
Choose $T_0>0$ such that $q(T_0)<1$. Then $(\Delta^{n})_{n\ge0}$ converges to $0$
in $L^2(\Omega;C([0,T_0];\mathbb{H}))$, so $(Y^{n})_{n\ge0}$ is Cauchy in that space.
Therefore there exists $Y\in L^2(\Omega;C([0,T_0];\mathbb{H}))$ such that
\[
Y^{n}\longrightarrow Y\quad \text{in }L^2(\Omega;C([0,T_0];\mathbb{H})).
\]

We pass to the limit in \eqref{eq:picard}.
Since $P_{t-s}$ is bounded on $\mathbb{H}$ and $F,G$ are Lipschitz, we have
\[
\int_0^t P_{t-s}F(s,Y^{n}(s))\,ds \to \int_0^t P_{t-s}F(s,Y(s))\,ds
\quad\text{in }L^2(\Omega;\mathbb{H}),
\]
by dominated convergence in Bochner spaces.
Similarly, using BDG and Lipschitz continuity,
\[
\int_0^t P_{t-s}G(s,Y^{n}(s))\,dB(s) \to \int_0^t P_{t-s}G(s,Y(s))\,dB(s)
\quad\text{in }L^2(\Omega;\mathbb{H}).
\]
Thus $Y$ satisfies \eqref{eq:mildagain} for all $t\in[0,T_0]$.

Let $N\in\N$ such that $N T_0\ge T$ and set $t_k:=kT_0\wedge T$.
Assuming $Y$ is constructed on $[0,t_k]$, we solve on $[t_k,t_{k+1}]$ the shifted
mild equation
\[
Y(t)=P_{t-t_k}Y(t_k)+\int_{t_k}^t P_{t-s}F(s,Y(s))\,ds+\int_{t_k}^t P_{t-s}G(s,Y(s))\,dB(s),
\]
by the same Picard argument. This yields a solution
on $[0,T]$.

\paragraph{Uniqueness.}
Let $Y$ and $\widetilde Y$ be two mild solutions on $[0,T_0]$ and set
$Z(t):=Y(t)-\widetilde Y(t)$. Subtracting their mild formulations and repeating the
estimate of Step 3 gives
\[
\E\sup_{t\le T_0}\|Z(t)\|_{\mathbb{H}}^2 \le q(T_0)\,\E\sup_{t\le T_0}\|Z(t)\|_{\mathbb{H}}^2,
\]
hence $Z\equiv 0$ on $[0,T_0]$ since $q(T_0)<1$. Iterating on each interval
$[t_k,t_{k+1}]$ yields $Y=\widetilde Y$ on $[0,T]$, proving pathwise uniqueness.

This completes the proof.
\end{proof}
%%%%%%%%%%%%%%%%%%%%%%Pb Game%%%%%%%%%%%%%%%%%%ù 
\section{The Game Problem}

This paper investigates a class of controlled SPDEs arising in the modeling of composite materials with
spatially heterogeneous properties. The controlled system is described by
\begin{equation}\label{eqcontrol.1}
\left\{
\begin{array}{rcl}
dY(t,x)
&=&\mathcal{A}^{u_1,u_2}_{\rho,a,b} Y(t,x)\, dt
+
\kappa(t,x,Y(t,x),u_1(t),u_2(t))\, dt
\\[0.2cm]
&&
+
\sigma(t,x,Y(t,x),u_1(t),u_2(t))\, dB(t),
\quad (t,x) \in [0,T]\times\mathbb{R},
\\[0.2cm]
Y(0,x) &=& \xi(x), \quad x \in \mathbb{R},
\end{array}
\right.
\end{equation}
where $Y(t,x)$ represents the state of the material at time $t$
and spatial position $x$, for instance temperature,
concentration, or mechanical stress.

\medskip

The processes $u_i(t)=u_i(t,\omega)$, $i=1,2$,
are control processes taking values in a given convex set
$U_i \subset \mathbb{R}$.
We assume that $u_1$ and $u_2$ are $\mathbb{F}$-predictable,
and that the control pair $(u_1,u_2)$ belongs to an admissible set
$\mathcal U_1 \times \mathcal U_2$.

\medskip

The operator $\mathcal{A}^{u_1,u_2}_{\rho,a,b}$
is defined through its action on a sufficiently smooth function
$\varphi$ by
\begin{equation}\label{operatorcontrol}
(\mathcal{A}^{u_1,u_2}_{\rho,a,b}\varphi)(x)
=
\frac{\rho(x)}{2}\,
\frac{d}{dx}
\left(
a(x,u_1)\frac{d\varphi}{dx}(x)
\right)
+
b(x,u_2)\frac{d\varphi}{dx}(x).
\end{equation}

This operator models diffusion and transport phenomena in a spatially
heterogeneous medium.

\begin{itemize}
\item
The term
\[
\frac{\rho(x)}{2}
\frac{d}{dx}
\left(
a(x,u_1)\frac{d\varphi}{dx}
\right)
\]
represents a diffusion operator in divergence form.
The function $\rho(x)$ acts as a spatial weight (for instance related to
density), while $a(x,u_1)$ denotes a diffusivity coefficient
depending on both the spatial position and the control $u_1$.
This term describes a diffusion process with spatially varying
and control-dependent intensity.

\item
The term
$
b(x,u_2)\dfrac{d\varphi}{dx}
$
represents a first-order transport (drift) component.
The function $b(x,u_2)$ models a velocity field depending on the
spatial position and on the control $u_2$.
\end{itemize}

Consequently, the operator
$\mathcal{A}^{u_1,u_2}_{\rho,a,b}$
describes a convection–diffusion dynamics in which
the diffusion intensity is influenced by the control $u_1$
and the transport mechanism is influenced by $u_2$,
while spatial heterogeneities are encoded through
$\rho(x)$ and the $x$-dependence of $a$ and $b$.

\medskip
We impose the following assumptions on the coefficients
$\kappa$, $\sigma$ and on the spatial functions
$\rho$, $a$, and $b$.

\begin{hypothesis}[Assumptions on the spatial coefficients]
\label{hypo_spatial}

Let $\rho : \mathbb R \to (0,\infty)$,
$a : \mathbb R \times U_1 \to (0,\infty)$,
and
$b : \mathbb R \times U_2 \to \mathbb R$
be measurable functions.

We assume the following uniform ellipticity and boundedness
conditions: there exist constants
$0 < \lambda \le \Lambda < \infty$
such that for all $x \in \mathbb R$,
$u_1 \in U_1$, and $u_2 \in U_2$,
\begin{align*}
&\lambda \le \rho(x) \le \Lambda, \\
\lambda \le a(x,u_1) \le &\Lambda,\quad \text{and}\quad 
|b(x,u_2)| \le \Lambda.
\end{align*}
\end{hypothesis}

\begin{hypothesis}[Assumptions on the coefficients]
\label{hypo1u}
The coefficients
\[
\kappa,\sigma :
\Omega \times [0,T] \times \mathbb R \times U_1 \times U_2
\longrightarrow
\mathbb R
\]
are measurable and globally Lipschitz continuous
with respect to the state variable $y$.

More precisely, there exists a constant $L>0$ such that for all
$(t,x) \in [0,T]\times\mathbb R$,
all $u_1\in U_1$, $u_2\in U_2$,
and all $y,y'\in\mathbb R$,
\begin{align*}
|\kappa(t,x,y,u_1,u_2)-\kappa(t,x,y',u_1,u_2)|
&\le L |y-y'|,\\
|\sigma(t,x,y,u_1,u_2)-\sigma(t,x,y',u_1,u_2)|
&\le L |y-y'|.
\end{align*}

Moreover, they satisfy the linear growth condition
\begin{align*}
|\kappa(t,x,y,u_1,u_2)|
&\le L(1+|y|),\\
|\sigma(t,x,y,u_1,u_2)|
&\le L(1+|y|),
\end{align*}
for all $(t,x,y,u_1,u_2)$.
\end{hypothesis}

In what follows, we use the notation $\mathcal{A}^{u_1,u_2}_{\rho,a,b} := \mathcal{A}_{(u_1,u_2)}.
$
\subsection{Objective Function and Hamiltonian System}

We consider a two-player stochastic differential game in a nonzero-sum setting.
Each player $i \in \{1,2\}$ seeks to minimize their own cost functional,
which depends on both controls.
The control pair $(u_1,u_2)$ belongs to an admissible set
$\mathcal U_1 \times \mathcal U_2$.

For each player $i$, the cost functional is defined by
\begin{align*}
J^i(u_1,u_2)
=
\mathbb E \Bigg[
\int_0^T \int_{\mathbb R}
L^i(t,x,Y(t,x),u_1(t),u_2(t))\,\frac{dx}{\rho(x)}\,dt
+
\int_{\mathbb R}
g^i(Y(T,x))\,\frac{dx}{\rho(x)}
\Bigg].
\end{align*}

Each player solves
\[
\text{Player } i:
\qquad
\min_{u_i\in\mathcal U_i}
J^i(u_1,u_2),
\quad
\text{given the other player's control}.
\]  

The state process $Y(t,x)$ evolves according to the controlled SPDE
\begin{align*}
dY(t,x)
&=
\Big[
\mathcal A_{(u_1,u_2)}Y(t,x)
+
\kappa(t,x,Y(t,x),u_1(t),u_2(t))
\Big]dt
\\ &\qquad +
\sigma(t,x,Y(t,x),u_1(t),u_2(t))\,dB(t),
\end{align*}
where $\mathcal A_{(u_1,u_2)}$ is a differential operator
depending on the control pair $(u_1,u_2)$.

The Hamiltonian $H^i$ associated with player $i$ is defined by
\begin{align*}
H^i(t,x,y,u_1,u_2,p^i,q^i)
&=
p^i(t,x)
\Big[
\mathcal A_{(u_1,u_2)} y
+
\kappa(t,x,y,u_1,u_2)
\Big]
\\
&\quad
+
q^i(t,x)\,
\sigma(t,x,y,u_1,u_2)
+
L^i(t,x,y,u_1,u_2).
\end{align*}
Let $(p^i,q^i)$ denote the adjoint processes.
The backward stochastic  partial differential equation BSPDE satisfied by $(p^i,q^i)$ is given by
\begin{align*}
\begin{cases}
-dp^i(t,x)
=
\displaystyle
\frac{\delta H^i}{\delta y}
\big(t,x,Y(t,x),u_1,u_2,p^i,q^i\big)\,dt
-
q^i(t,x)\,dB(t),
\\[0.3cm]
p^i(T,x)
=
g_y^i(Y(T,x)),
\end{cases}
\end{align*}
where $\dfrac{\delta H^i}{\delta y}$ denotes the Fréchet derivative
of $H^i$ with respect to the state variable $y$ in
$\mathbb{H}$.
A direct computation yields:
\begin{align}
\frac{\delta H^i}{\delta y}
=
\mathcal A^*_{(u_1,u_2)} p^i
+
p^i\,\kappa_y
+
q^i\,\sigma_y
+
L_y^i,\end{align}
where $\mathcal A^*_{(u_1,u_2)}$
denotes the adjoint of $\mathcal A_{(u_1,u_2)}$
in $\mathbb{H}$.

Therefore the adjoint equation can be written explicitly as
\begin{align*}
\begin{cases}
-dp^i(t,x)
=
\Big[
\mathcal A^*_{(u_1,u_2)} p^i(t,x)
+
p^i(t,x)\kappa_y
+
q^i(t,x)\sigma_y
+
L_y^i
\Big]dt
-
q^i(t,x)\,dB(t),
\\[0.3cm]
p^i(T,x)
=
g_y^i(Y(T,x)).
\end{cases}
\end{align*}

We assume that, for each $i=1,2$, the functions 
$L^i(t,x,y,u_1,u_2)$ and $g^i(x,y)$ are measurable in 
$(\omega,t,x)$, continuously differentiable with respect to $y$, 
and satisfy standard Lipschitz and linear growth conditions 
ensuring well-posedness of the state and adjoint equations.

To analyze the variational structure, we define the reduced Hamiltonian
as a functional on $\mathbb{H}$ by
\[
h^i(\varphi)
:=
H^i\big(t,x,\varphi,u_1,u_2,p^i,q^i\big),
\]
where $H^i$ depends on $\varphi$ through
$\mathcal A_{(u_1,u_2)}\varphi$ and $\varphi$ itself.

Assuming that $h^i$ is Fréchet differentiable at $\varphi$,
we have the first-order expansion
\[
h^i(\varphi+\psi)
=
h^i(\varphi)
+
\langle D_\varphi h^i(\varphi),\psi\rangle_{\mathbb{H}}
+
o(\|\psi\|_{\mathbb{H}}).
\]

In particular, if $H^i$ depends on $\varphi$
only through the term
$p^i\,\mathcal A_{(u_1,u_2)}\varphi$,
then the directional derivative satisfies
\[
\langle D_\varphi h^i(\varphi),\psi\rangle_{\mathbb{H}}
=
\langle
\mathcal A_{(u_1,u_2)}\psi,
\, p^i
\rangle_{\mathbb{H}}.
\]

By duality, this can equivalently be written as
\[
\langle
\psi,
\, \mathcal A^*_{(u_1,u_2)} p^i
\rangle_{\mathbb{H}}.
\]

%%%%%%%%%%%%%%%%%%%%%%%%%%%%%%Sufficient%%%%%%%%%%%%%%%%%%%%%%%%%%%%%%%%%%%%%%%
\subsection{Sufficient Optimality Condition for Games}

In this section, we derive a sufficient condition under which a given pair of control strategies constitutes a Nash equilibrium in a two-player nonzero-sum stochastic differential game. Our approach extends the classical stochastic maximum principle to the game setting, where each player seeks to minimize their own cost functional subject to a shared controlled state equation. The condition is formulated in terms of the convexity of the Hamiltonians and terminal costs, along with a pointwise minimum condition on the Hamiltonian with respect to each player's control, while fixing the strategy of the opponent. We show that, under these assumptions, no player has an incentive to unilaterally deviate from their strategy, thus verifying optimality in the Nash equilibrium sense.
\begin{theorem}[Sufficient Maximum Principle for Nonzero-Sum Game]\label{thm:suff_game}
Let \( (\widehat{u}_1(t), \widehat{u}_2(t)) \in \mathcal{U}_1 \times \mathcal{U}_2 \) be a candidate Nash equilibrium, with associated state process \( \widehat{Y}(t,x) \) solving the state equation under \( (\widehat{u}_1, \widehat{u}_2) \), and let \( (\widehat{p}^i(t,x), \widehat{q}^i(t,x)) \) be the corresponding adjoint processes solving the BSPDE for player \( i = 1,2 \).

Assume the following conditions hold for each \( i = 1,2 \):
\begin{enumerate}
    \item For a.e. $(t,x)$, the Hamiltonian
\[
(\varphi,u_1,u_2)\longmapsto
H^i\big(t,x,\mathcal A_{(u_1,u_2)}(\varphi),\varphi,u_1,u_2,p^i,q^i\big)
\]
is convex in $(\varphi,u_1,u_2)$.
    
    \item The terminal cost \( g^i(y) \) is convex.

    \item The controls \( \widehat{u}_i(t), i=1,2 \) satisfy the \emph{minimum conditions}:
    \begin{enumerate}
    \item $ \di 
    \widehat{u}_1(t) \in \arg\min_{u_1 \in U_1} \int_{\mathbb{R}} H^1(t,x,\mathcal{A}_{(u_1,\widehat{u}_{2})} \widehat{Y}(t,x), \widehat{Y}(t,x), u_1, \widehat{u}_{2}(t), \widehat{p}^1(t,x), \widehat{q}^1(t,x)) \dfrac{dx}{\rho(x)},
    $\\
    for a.e. \( t \in [0,T] \), where \( \widehat{u}_{2} \) denotes the fixed control of player $2$.
    \item $ \di 
    \widehat{u}_2(t) \in \arg\min_{u_2 \in U_2} \int_{\mathbb{R}} H^2(t,x,\mathcal{A}_{(\widehat{u}_1,u_{2})} \widehat{Y}(t,x), \widehat{Y}(t,x), \widehat{u}_1, u_{2}(t), \widehat{p}^2(t,x), \widehat{q}^2(t,x)) \dfrac{dx}{\rho(x)},
    $\\
    for a.e. \( t \in [0,T] \), where \( \widehat{u}_{1} \) denotes the fixed control of player $1$.
    \end{enumerate}

\end{enumerate}
Then, the pair \( (\widehat{u}_1, \widehat{u}_2) \) forms a Nash equilibrium.
\end{theorem}

\begin{proof}Fix any admissible control $u_1\in\mathcal U_1$ and consider player~1.
Let $Y(t,x)$ denote the state process associated with $(u_1,\widehat u_2)$.
We compare $J^1(u_1,\widehat u_2)$ with $J^1(\widehat u_1,\widehat u_2)$.
Set $\widetilde Y := Y - \widehat Y$.

We decompose
\[
J^1(u_1,\widehat u_2)
-
J^1(\widehat u_1,\widehat u_2)
=
I_1 + I_2,
\]
where
\begin{align*}
I_1
&=
\mathbb E
\int_0^T
\int_{\mathbb R}
\big(
L^1(t,x,Y,u_1,\widehat u_2)
-
L^1(t,x,\widehat Y,\widehat u_1(t),\widehat u_2(t))
\big)
\frac{dx}{\rho(x)} dt,\\
I_2
&=
\mathbb E
\int_{\mathbb R}
\big(
g^1(Y(T,x))
-
g^1(\widehat Y(T,x))
\big)
\frac{dx}{\rho(x)}.
\end{align*}
By convexity of $g^1$,
\[
I_2
\ge
\mathbb E
\int_{\mathbb R}
\widehat p^1(T,x)\,\widetilde Y(T,x)
\frac{dx}{\rho(x)}.
\]

Define
\[
\widetilde H^1(t,x)
=
H^1(t,x,Y,u_1,\widehat u_2,\widehat p^1,\widehat q^1)
-
H^1(t,x,\widehat Y,\widehat u_1,\widehat u_2,
\widehat p^1,\widehat q^1).
\]

Using the definition of the Hamiltonian,
\[
H^1
=
p^1(\mathcal A_{(u_1,u_2)}y+\kappa)
+
q^1\sigma
+
L^1,
\]
we can rewrite $I_1$ in Hamiltonian form and obtain
\[
I_1
=
\mathbb E
\int_0^T
\int_{\mathbb R}
\Big[
\widetilde H^1
-
\widehat p^1\big(\mathcal A_{(u_1,\widehat u_2)}Y
-
\mathcal A_{(\widehat u_1,\widehat u_2)}\widehat Y
+
\widetilde\kappa^1
\big)
-
\widehat q^1\widetilde\sigma^1
\Big]
\frac{dx}{\rho(x)} dt.
\]
with $\widetilde\sigma$ and $\widetilde\kappa$ denote the following notations:
\begin{align*}
\widetilde{\kappa}^1(t,x)
&=
\kappa(t,x,Y(t,x),u_1(t),\widehat u_2(t))
-
\kappa(t,x,\widehat Y(t,x),\widehat u_1(t),\widehat u_2(t)),\\
\widetilde{\sigma}^1(t,x)
&=
\sigma(t,x,Y(t,x),u_1(t),\widehat u_2(t))
-
\sigma(t,x,\widehat Y(t,x),\widehat u_1(t),\widehat u_2(t)).
\end{align*}

Since the Hamiltonian contains the operator term,
the adjoint equation is written in functional-derivative form:
\[
-d\widehat p^1
=
\frac{\delta H^1}{\delta y}
\big(t,x,\widehat Y,\widehat u_1,\widehat u_2,
\widehat p^1,\widehat q^1\big)\,dt
-
\widehat q^1\,dB.
\]

Applying It\^{o}'s formula to
$\widehat p^1(t,x)\widetilde Y(t,x)$,
integrating in $x$ with weight $\rho^{-1}(x)$,
and using the adjoint relation in
$\mathbb{H}$,
we obtain cancellation of all state-variation terms and deduce
\[
J^1(u_1,\widehat u_2)
-
J^1(\widehat u_1,\widehat u_2)
\ge
\mathbb E
\int_0^T
\int_{\mathbb R}
\Big[
\widetilde H^1
-
\langle D_\varphi H^1,\widetilde Y\rangle
\Big]
\frac{dx}{\rho(x)} dt.
\]

By convexity of $H^1$ in $(\varphi,u_1)$,
\[
\widetilde H^1
\ge
\langle D_\varphi H^1,\widetilde Y\rangle
+
\frac{\partial H^1}{\partial u_1}
(u_1-\widehat u_1).
\]

Therefore,
\[
J^1(u_1,\widehat u_2)
-
J^1(\widehat u_1,\widehat u_2)
\ge
\mathbb E
\int_0^T
\int_{\mathbb R}
\frac{\partial H^1}{\partial u_1}
(t,x)
(u_1-\widehat u_1(t))
\frac{dx}{\rho(x)} dt.
\]

By the minimum condition,
\[
\int_{\mathbb R}
\frac{\partial H^1}{\partial u_1}(t,x)
(u_1-\widehat u_1(t))
\frac{dx}{\rho(x)}
\ge 0,
\]
hence
\[
J^1(u_1,\widehat u_2)
\ge
J^1(\widehat u_1,\widehat u_2).
\]

The same argument applies to player $2$.
Therefore $(\widehat u_1,\widehat u_2)$ is a Nash equilibrium.
\end{proof}

%%%%%%%%%%%%%%%%Necessary cond%%%%%%%%%%%%%%%%%%%%%%%%
\subsection{Necessary Optimality Condition for Games}

In this section we establish a necessary maximum principle for the
two-player nonzero-sum stochastic differential game.
In contrast to the sufficient maximum principle, no convexity
assumptions are imposed. Instead, we derive first-order optimality
conditions by means of variational arguments.

The key idea is to perturb one player's control while keeping the
other player's control fixed, and to analyze the first variation of
the corresponding cost functional.

\begin{itemize}
\item
For each player $i\in\{1,2\}$ and for every $t\in[0,T]$, let
$\beta_{i,0}(t)$ be a bounded $\mathcal F_t$-measurable random variable
satisfying $\|\beta_{i,0}(t)\|\le K$.
Define
\[
\delta_i(t)
=
\frac{1}{2K}\,\text{dist}\!\big(u_i(t),\partial U_i\big)\wedge 1
\;>\;0,
\qquad
\beta_i(t)
=
\delta_i(t)\,\beta_{i,0}(t).
\]
For any $a\in(-1,1)$, the perturbed control for player $i$ is given by
\[
\tilde u_i(t)
=
u_i(t)+a\,\beta_i(t).
\]
By construction, $\tilde u_i(t)\in U_i$ almost surely, and therefore
$\tilde u_i\in\mathcal U_i$.
\end{itemize}

\begin{theorem}[Necessary Maximum Principle for Games]
\label{thm:necessary_game}
Let $(\hat u_1,\hat u_2)\in\mathcal U_1\times\mathcal U_2$
be a Nash equilibrium and let $\widehat Y$
be the corresponding state.
For each $i\in\{1,2\}$ let
$(\widehat p^i,\widehat q^i)$ denote the associated adjoint processes.

Then the following are equivalent:

\begin{enumerate}

\item For every bounded admissible direction $\beta_i$
such that $\hat u_i+a\beta_i\in\mathcal U_i$
for all sufficiently small $a$,
\[
\left.\frac{d}{da}
J^i(\hat u_i+a\beta_i,\hat u_{-i})
\right|_{a=0}
=
0.
\]

\item (Variational inequality.)
For almost every $t\in[0,T]$, almost surely,
\[
\int_{\R}
\partial_{u_i}H^i
\big(t,x,\widehat Y(t,x),\hat u_1(t),\hat u_2(t),
\widehat p^i(t,x),\widehat q^i(t,x)\big)
(v_i-\hat u_i(t))
\frac{dx}{\rho(x)}
\ge 0,
\quad
\forall v_i\in U_i.
\]

If $U_i$ is open (or $\hat u_i(t)\in\mathrm{int}(U_i)$ a.s.),
this reduces to
\[
\partial_{u_i}H^i
\big(t,x,\widehat Y(t,x),\hat u_1(t),\hat u_2(t),
\widehat p^i(t,x),\widehat q^i(t,x)\big)
=
0.
\]
\end{enumerate}
\end{theorem}

\begin{proof}
Fix $i\in\{1,2\}$ and denote $\hat u:=(\hat u_i,\hat u_{-i})$.
Let $\beta_i$ be a bounded progressively measurable process such that
$\hat u_i+a\beta_i\in\mathcal U_i$
for all sufficiently small $|a|$.
Define
\[
\tilde u_i^a:=\hat u_i+a\,\beta_i,
\qquad
Y^a:=Y^{(\tilde u_i^a,\hat u_{-i})},
\]
and set $\widehat Y:=Y^{(\hat u_i,\hat u_{-i})}$.

Define the first variation
\[
Z^i(t,x)
=
\left.
\frac{d}{da}
Y^{(\hat u_i+a\beta_i,\hat u_{-i})}(t,x)
\right|_{a=0},
\]
assuming the derivative exists in $\mathbb{H}$.

Differentiating the state equation at $a=0$ yields
\[
\left\{
\begin{aligned}
dZ^i(t,x)
&=
\Big[
\mathcal A_{\hat u} Z^i(t,x)
+
\partial_y\kappa(\hat\Xi(t,x)) Z^i(t,x)
+
\partial_{u_i}\kappa(\hat\Xi(t,x))\,\beta_i(t)
\\
&\qquad
+
(\partial_{u_i}\mathcal A_{\hat u}[\beta_i(t)])\widehat Y(t,x)
\Big]dt
\\
&\quad
+
\Big[
\partial_y\sigma(\hat\Xi(t,x)) Z^i(t,x)
+
\partial_{u_i}\sigma(\hat\Xi(t,x))\,\beta_i(t)
\Big]dB(t),
\\
Z^i(0,x)&=0,
\end{aligned}
\right.
\]
where
\[
\hat\Xi(t,x)
:=
\hat\Xi(t,x,\widehat Y(t,x),\hat u_1(t),\hat u_2(t)),
\]
and
$\partial_{u_i}\mathcal A_{\hat u}[\beta_i(t)]$
denotes the directional derivative of the operator
$\mathcal A_{(u_1,u_2)}$
with respect to $u_i$ at $\hat u=(\hat u_1,\hat u_2)$
in the direction $\beta_i(t)$,
acting on $\widehat Y(t,\cdot)$.

Differentiating the cost functional and using dominated convergence gives
\begin{align*}
\left.\frac{d}{da}
J^i(\hat u_i+a\beta_i,\hat u_{-i})
\right|_{a=0}
&=
\E\Bigg[
\int_0^T\!\!\int_{\R}
\Big(
\partial_y L^i(\hat\Xi) Z^i
+
\partial_{u_i}L^i(\hat\Xi)\beta_i
\Big)
\frac{dx}{\rho(x)} dt
\\
&\qquad
+
\int_{\R}
\partial_y g^i(\widehat Y(T,x)) Z^i(T,x)
\frac{dx}{\rho(x)}
\Bigg].
\end{align*}

Since the Hamiltonian is defined by
\[
H^i(t,x,y,u_1,u_2,p,q)
=
p(\mathcal A_{(u_1,u_2)}y+\kappa)
+
q\sigma
+
L^i,
\]
the adjoint equation must be written in functional derivative form:
%Let $(\widehat p^i,\widehat q^i)$ solve the adjoint equation
\[
\left\{
\begin{aligned}
-d\widehat p^i(t,x)
&=
\frac{\delta H^i}{\delta y}
(t,x,\widehat Y,\hat u_1,\hat u_2,
\widehat p^i,\widehat q^i)\,dt
-
\widehat q^i(t,x)\,dB(t),
\\
\widehat p^i(T,x)
&=
\partial_y g^i(\widehat Y(T,x)).
\end{aligned}
\right.
\]

Apply It\^{o}'s formula to
$\widehat p^i(t,x)Z^i(t,x)$,
integrate in $x$ with weight $\rho^{-1}(x)$,
and use the adjoint relation in
$\mathbb{H}$:
\[
\langle \mathcal A_{\hat u}Z^i,\widehat p^i\rangle
=
\langle Z^i,\mathcal A_{\hat u}^*\widehat p^i\rangle.
\]

All $Z^i$-terms cancel, and we obtain
\begin{align*}
\left.\frac{d}{da}
J^i(\hat u_i+a\beta_i,\hat u_{-i})
\right|_{a=0}
&=
\E
\int_0^T\!\!\int_{\R}
\Big(
\partial_{u_i}L^i(\hat\Xi)
+
\widehat p^i\,\partial_{u_i}\kappa(\hat\Xi)
\\
&\qquad
+
\widehat q^i\,\partial_{u_i}\sigma(\hat\Xi)
+
\widehat p^i\,
(\partial_{u_i}\mathcal A_{\hat u}[\beta_i])\widehat Y
\Big)
\frac{dx}{\rho(x)} dt.
\end{align*}

By definition of the Hamiltonian, we have
\[
\partial_{u_i}H^i
=
\partial_{u_i}L^i
+
\widehat p^i\,\partial_{u_i}\kappa
+
\widehat q^i\,\partial_{u_i}\sigma
+
\widehat p^i\,\partial_{u_i}\mathcal A_{\hat u}\widehat Y.
\]

Therefore,
\[
\left.\frac{d}{da}
J^i(\hat u_i+a\beta_i,\hat u_{-i})
\right|_{a=0}
=
\E
\int_0^T\!\!\int_{\R}
\partial_{u_i}H^i
(t,x,\widehat Y,\hat u_1,\hat u_2,
\widehat p^i,\widehat q^i)
\beta_i(t)
\frac{dx}{\rho(x)} dt.
\]

Consequently, since $(\hat u_1,\hat u_2)$ is a Nash equilibrium,
the first variation must vanish for all admissible directions $\beta_i$.
This is equivalent to the variational inequality
\[
\int_{\R}
\partial_{u_i}H^i
(t,x,\widehat Y,\hat u_1,\hat u_2,
\widehat p^i,\widehat q^i)
(v_i-\hat u_i(t))
\frac{dx}{\rho(x)}
\ge 0,
\]
for all $v_i\in U_i$.
If $U_i$ is open, this reduces to the pointwise stationarity condition
$\partial_{u_i}H^i=0.$ This completes the proof. 
\end{proof}
\subsection{BSPDE : Existence and uniqueness}

In this section, we investigate the well-posedness of the backward SPDE associated with the adjoint system. Our main
objective is to establish the existence and uniqueness of adapted solutions
under suitable structural and integrability assumptions.

The analysis relies on classical arguments for backward stochastic evolution
equations in Hilbert spaces, following the framework developed in \cite{OksendalProskeZhang2005}. In particular, the proof is based on a priori estimates
and standard fixed-point arguments for linear BSPDEs.

To this end, we first introduce a set of additional hypotheses on the
coefficients of the equation, ensuring measurability, Lipschitz continuity,
and appropriate growth conditions. These assumptions allow us to place the
adjoint BSPDE within a well-established existence and uniqueness theory.

\begin{hypothesis}[Additional assumptions for $(\kappa,\sigma)$]
\label{hypoAdj}

In addition to Hypothesis~\ref{hypo1u}, we assume that the coefficients
$\kappa$ and $\sigma$ are continuously differentiable with respect to the
state variable $y$. Moreover, the partial derivatives
\[
\kappa_y(t,x,y,u_1,u_2),
\qquad
\sigma_y(t,x,y,u_1,u_2)
\]
exist and are uniformly bounded. Finally, the adjoint drift term is affine
with respect to the adjoint variables $(p^i,q^i)$.
\end{hypothesis}

\begin{theorem}[Existence and uniqueness of a BSPDE (general case)]
\label{thm:oks-general}

Let $V$ and $\mathbb{H}$ be separable Hilbert spaces such that
\[
V \subset \mathbb{H} \simeq \mathbb{H}^* \subset V^*
\]
with continuous and dense embeddings, and let $\mathcal L_{\rho,a,b}:V\to V^*$ be a linear operator.
Let $\phi\in L^2(\Omega,\mathcal F_T;\mathbb{H})$ and consider the BSPDE
\begin{equation}\label{eq:bspde-general}
\left\{
\begin{aligned}
-dP(t) &= \big[\mathcal L_{\rho,a,b} P(t) + b(t,P(t),Q(t))\big]\,dt - Q(t)\,dB(t),
\qquad t\in[0,T],\\
P(T) &= \phi,
\end{aligned}
\right.
\end{equation}
understood in the variational sense in $V^*$.

Assume that $b:\Omega\times[0,T]\times \mathbb{H}\times \mathbb{H}\to \mathbb{H}$ is progressively measurable and that
there exists a constant $c>0$ such that for all $t\in[0,T]$ and all
$(p_1,q_1),(p_2,q_2)\in \mathbb{H}\times \mathbb{H}$,
\begin{equation}\label{eq:lip-b}
\|b(t,p_1,q_1)-b(t,p_2,q_2)\|_\mathbb{H}
\le c\big(\|p_1-p_2\|_\mathbb{H}+\|q_1-q_2\|_\mathbb{H}\big),
\end{equation}
and
\begin{equation}\label{eq:int-b}
\mathbb E\!\left[\int_0^T \|b(t,0,0)\|_\mathbb{H}^2\,dt\right]<\infty.
\end{equation}

Then there exists a unique adapted solution $(P,Q)$ to \eqref{eq:bspde-general} such that
\[
P\in L^2(\Omega;C([0,T];\mathbb{H}))\cap L^2(\Omega\times(0,T);V),
\qquad
Q\in L^2(\Omega\times(0,T);\mathbb{H}),
\]
and for all $t\in[0,T]$ the following identity holds $\mathbb P$-a.s.:
\[
P(t)=\phi+\int_t^T \big(\mathcal L_{\rho,a,b} P(s)+b(s,P(s),Q(s))\big)\,ds-\int_t^T Q(s)\,dB(s).
\]
\end{theorem}

\begin{remark}
The adjoint BSPDE in Theorem~\ref{thm:oks-general} is a particular case of
\eqref{eq:bspde-general} with $P=p^i$ and $Q=q^i$, where the operator
$\mathcal L_{\rho,a,b}$ is given by the adjoint operator associated to
$\mathcal A_{\rho,a,b}.$
More precisely, the drift term takes the form
\[
b(t,p,q)
=
\kappa_y(t)\,p
+
\sigma_y(t)\,q
+
L_y^i(t).
\]
Under Hypothesis~\ref{hypoAdj}, the coefficients $\kappa_y$ and $\sigma_y$ are
uniformly bounded, so that the Lipschitz condition \eqref{eq:lip-b} holds with
\[
c=\|\kappa_y\|_\infty+\|\sigma_y\|_\infty.
\]
Moreover, the integrability condition \eqref{eq:int-b} reduces to
$L_y^i\in L^2(\Omega\times(0,T);\mathbb{H})$.
\end{remark}

\begin{proof}
The adjoint BSPDE is a backward stochastic evolution equation in the Gelfand
triple $V\subset \mathbb{H}\simeq \mathbb{H}^*\subset V^*$. Under Hypotheses~\ref{hypo1} and
\ref{hypoAdj}, the operator $\mathcal A^*_{(u_1,u_2)}$ satisfies the standard
coercivity/continuity assumptions, and the drift term is affine in
$(p^i,q^i)$ with progressively measurable bounded coefficients. Moreover,
$L_y^i\in L^2(\Omega\times(0,T);\mathbb{H})$ and
$\mathbb E\|g_y^i(Y(T))\|_\mathbb{H}^2<\infty$.
Hence the assumptions of Theorem~4.1 in \cite{OksendalProskeZhang2005} are satisfied,
and the existence and uniqueness of an adapted solution
$(p^i,q^i)$ in the stated spaces follow. The integral identity is obtained
by integrating the equation from $t$ to $T$.
\end{proof}

\subsection{Closed form for linear BSPDE}
\subsubsection{Smooth operator}

To illustrate the general theory, we consider a smooth constant coefficient setting in which the coefficients $(\rho,a,b)$ of the operator are fixed real numbers. In this case, an explicit representation of the solution can be obtained. For related results, we refer to \cite{MaYong}.

\medskip

Let $\rho_0,a_0,b_0,c,\gamma \in \mathbb{R}$ satisfy $\rho_0 a_0>0$, and define
\[
\kappa := \frac{\rho_0 a_0}{2},
\qquad
\sigma := \sqrt{\rho_0 a_0}.
\]
And, let $(W_t)_{t\ge0}$ be a one-dimensional Brownian motion on a filtered probability space
$(\Omega,\mathcal F,\mathbb{F},\mathbb P)$.

\begin{Example}\label{exBSPDE-Smooth}
Let $c,\gamma \in \mathbb{R}$. Consider the following BSPDE
\begin{equation}\label{BSPDE-main-f}
\begin{cases}
du(t,x)
=
-\Big[
\kappa u_{xx}(t,x)
-
b_0 u_x(t,x)
+
c\,u(t,x)
+
\gamma\, q(t,x)
+
f(t,x)
\Big]dt
+
q(t,x)\, dW_t,\\
u(T,x)=g(x),
\end{cases}
\end{equation}
for any $(t,x)\in [0,T]\times \R.$ Here, the functions $g$ and $f$ are assumed to be bounded and measurable. 
\end{Example}

\begin{theorem}[Explicit solution with $q$-drift and source term] \label{thm:explicitBSPDE}

The BSPDE defined in \eqref{BSPDE-main-f} admits a unique adapted classical solution satisfies
\begin{equation}\label{representationsmoothcase}
u(t,x)
=
\mathbb E^{\mathbb Q}
\Bigg[
e^{c(T-t)} g(X_T^{t,x})
+
\int_t^T
e^{c(s-t)} f(s,X_s^{t,x})
\, ds
\;\Bigg|\; \mathcal F_t
\Bigg],
\end{equation}
where under $\mathbb Q$ the forward diffusion SDE  solves
\[
dX_s
=
-(b_0+\sigma\gamma)\,ds
+
\sigma\, dW_s^{\mathbb Q},
\qquad X_t=x,
\]
with $dW_s^{\mathbb Q}=dW_s+\gamma ds.$
In particular, the solution of the forward can be written explicitly as
\[
X_T=x-(b_0+\sigma\gamma)(T-t)+
\sigma\big(W_T^{\mathbb Q}-W_t^{\mathbb Q}\big).
\]

Moreover, the associated martingale term satisfies
\begin{equation}
q(t,x)=\sigma\,u_x(t,x).    
\end{equation}
\end{theorem}

\begin{proof}
Define the forward diffusion under $\mathbb P$:
\[
dX_s = -b_0\, ds + \sigma\, dW_s,
\qquad X_t=x.
\]
Set
\[
Y_s=u(s,X_s),
\qquad
Z_s=\sigma u_x(s,X_s)+q(s,X_s).
\]
Applying It\^o--Ventzel yields
\[
dY_s
=
-
\Big[
cY_s
+
\gamma q(s,X_s)
+
f(s,X_s)
\Big]ds
+
Z_s dW_s.
\]
Substituting $q=Z-\sigma u_x$ gives
\[
dY_s
=
-
\Big[
cY_s
+
\gamma Z_s
-
\gamma\sigma u_x(s,X_s)
+
f(s,X_s)
\Big]ds
+
Z_s dW_s.
\]
Introduce the probability measure $\mathbb Q$ via
\[
\frac{d\mathbb Q}{d\mathbb P}
=
\exp\!\Big(
-\gamma W_T
-\tfrac12\gamma^2 T
\Big),
\]
so that
$
dW_s^{\mathbb Q}=dW_s+\gamma ds.
$
Under $\mathbb Q$, the $\gamma Z_s$ term cancels and we obtain
\[
dY_s
=
-
\Big[
cY_s
+
f(s,X_s)
\Big]ds
+
Z_s dW_s^{\mathbb Q}.
\]
Using the integrating factor $e^{\int c}$ and taking conditional expectation under $\mathbb Q$ yields
\[
u(t,x)
=
\mathbb E^{\mathbb Q}
\Bigg[
e^{c(T-t)} g(X_T)
+
\int_t^T e^{c(s-t)} f(s,X_s)\, ds
\;\Bigg|\; \mathcal F_t
\Bigg].
\]
This establishes the first part of the proof.

On the one hand, from the definition of $Z$, we have
\[
Z_t
=
\sigma u_x(t,X_t)
+
q(t,X_t).
\]
On the other hand, since $Y_t=u(t,X_t)$, the martingale representation (Clark--Ocone formula under $\mathbb Q$) gives
\[
Z_t=D_t^{\mathbb Q}Y_t.
\]
Using the chain rule of Malliavin calculus,
\[
D_t^{\mathbb Q}Y_t
=
D_t^{\mathbb Q}u(t,X_t)
=
\sigma u_x(t,X_t)
+
D_t^{\mathbb Q}u(t,X_t).
\]
Therefore, we get
\[
Z_t
=
\sigma u_x(t,X_t)
+
D_t^{\mathbb Q}u(t,X_t).
\]
Comparing with the definition of $Z_t$ yields
\[
q(t,X_t)=D_t^{\mathbb Q}u(t,X_t).
\]
Since this holds for all $x$, we conclude
\[
q(t,x)=D_t^{\mathbb Q}u(t,x).
\]
In the present constant coefficient Markov setting,
\[
D_t^{\mathbb Q}u(t,x)=\sigma u_x(t,x),
\]
and hence
$
q(t,x)=\sigma u_x(t,x),
$
which completes the proof.
\end{proof}

We now consider a more specific case where $f \equiv 0$. This corresponds to a classical subclass of BSPDEs arising in stochastic control (see, e.g., \cite{Oksendal}). In this setting, the structure of the equation allows for a more explicit computation, which will be presented below.

\begin{corollary}[Particular case]\label{corBSPDE}

If $f\equiv 0$, then  the random variable $X_T$ is Gaussian and satisfies
\[
X_T \sim \mathcal N\!\left(
x-(b_0+\sigma\gamma)(T-t),
\;\sigma^2 (T-t)
\right).
\]
Consequently, the solution admits the explicit representation
\[
u(t,x)
=
e^{c(T-t)}
\int_{\mathbb R}
\frac{g(y)}{\sqrt{2\pi\sigma^2 (T-t)}}
\exp\!\left(
-
\frac{\big(y-x+(b_0+\sigma\gamma)(T-t)\big)^2}
{2\sigma^2 (T-t)}
\right)
dy.
\]
Moreover, the associated process $q$ is given by
\[
q(t,x)
=
e^{c\tau}
\int_{\mathbb R}
g(y)
\frac{y-m}{\sigma\tau}
\frac{1}{\sqrt{2\pi\sigma^2\tau}}
\exp\!\left(
-\frac{(y-m)^2}{2\sigma^2\tau}
\right)
dy.
\]
\end{corollary}

\begin{proof}
Assume $f\equiv 0$, Let $\tau:=T-t$ and define
\[
m:=x-(b_0+\sigma\gamma)\tau.
\]
Then the solution can be written as
\[
u(t,x)
=
e^{c\tau}
\int_{\mathbb R}
g(y)
\frac{1}{\sqrt{2\pi\sigma^2\tau}}
\exp\!\left(
-\frac{(y-m)^2}{2\sigma^2\tau}
\right)
dy.
\]
Since $\partial_x m=1$, differentiation under the integral sign yields
\[
\partial_x
\exp\!\left(
-\frac{(y-m)^2}{2\sigma^2\tau}
\right)
=
\frac{y-m}{\sigma^2\tau}
\exp\!\left(
-\frac{(y-m)^2}{2\sigma^2\tau}
\right).
\]
Hence, we obtain
\[
u_x(t,x)
=
e^{c\tau}
\int_{\mathbb R}
g(y)
\frac{y-m}{\sigma^2\tau}
\frac{1}{\sqrt{2\pi\sigma^2\tau}}
\exp\!\left(
-\frac{(y-m)^2}{2\sigma^2\tau}
\right)
dy.
\]
Since we have proved that
$
q(t,x)=\sigma u_x(t,x),
$
we obtain, therefore, 
\[
q(t,x)
=
e^{c\tau}
\int_{\mathbb R}
g(y)
\frac{y-m}{\sigma\tau}
\frac{1}{\sqrt{2\pi\sigma^2\tau}}
\exp\!\left(
-\frac{(y-m)^2}{2\sigma^2\tau}
\right)
dy.
\]

Let $Z\sim\mathcal N(0,1)$ and note that
$
X_T = m + \sigma\sqrt{\tau}\,Z.$ It implies, then
\[
u(t,x)=e^{c\tau}\mathbb E[g(X_T)].
\]
Differentiating under the expectation gives
\[
u_x(t,x)
=
e^{c\tau}
\mathbb E\!\left[
g(X_T)\frac{Z}{\sigma\sqrt{\tau}}
\right],
\]
and therefore
$
q(t,x)=
e^{c\tau}
\mathbb E\!\left[
g(X_T)\frac{Z}{\sqrt{\tau}}
\right].
$
Equivalently, it implies
\[
q(t,x)
=
\frac{e^{c\tau}}{\sigma\tau}
\mathbb E\!\left[
g(X_T)(X_T-m)
\right].
\]
Hence, the proof is achieved. 
\end{proof}

\subsubsection{Non smooth Operator}
We now present an example illustrating the general results in the case where the coefficients $(\rho,a,b)$ are piecewise constant functions of the form \eqref{coeff-piecewise}.

\medskip 
Define
\[
\kappa(x)=\frac{\rho(x)a(x)}{2},
\qquad
\sigma(x)=\sqrt{\rho(x)a(x)}.
\]
And, let $
\di \kappa^\pm=\frac{\rho^\pm a^\pm}{2}
\,\text{and}\,
\sigma^\pm=\sqrt{\rho^\pm a^\pm},
$ are two positives constants.

\begin{Example}\label{EX_BSPDE-piecewise}
   Consider the BSPDE
\begin{equation}\label{BSPDE-piecewise}
\begin{cases}
du(t,x)
=
-\Big[
\kappa(x) u_{xx}(t,x)
-
b(x) u_x(t,x)
+
c\,u(t,x)
+
\gamma\, q(t,x)
+
f(t,x)
\Big]dt
+
q(t,x)\, dW_t,\\
u(T,x)=g(x),
\end{cases}
\end{equation} 
for any $(t,x)\in [0,T]\times \R.$  Here, $f,g$ are given bounded measurable functions.
\end{Example}
\begin{theorem}[Representation for piecewise constant coefficients]\label{thm:piecewise-BSPDE}
Consider the BSPDE in Example \eqref{EX_BSPDE-piecewise}. Let $X^{t,x}$ be the solution of the forward SDE
\begin{equation}\label{forward-piecewise}
dX_s
=
-
\big(b(X_s)+\sigma(X_s)\gamma\big)\,ds
+
\sigma(X_s)\, dW_s^{\mathbb Q},
\qquad X_t=x,
\end{equation}
where $W^{\mathbb Q}$ is defined by
$
dW_s^{\mathbb Q}=dW_s+\gamma\,ds.
$
Then the following properties hold:
\begin{enumerate}
\item[(i)] For $x\neq 0$, on each half-line, the function $u$ is a classical solution of
\[
u_t+\kappa^\pm u_{xx}
-(b^\pm+\sigma^\pm\gamma)u_x+cu+f=0
\quad \text{in }(0,T)\times\mathbb R_\pm.
\]

\item[(ii)] The solution satisfies the transmission conditions at $x=0$:
\[
u(t,0^-)=u(t,0^+),
\qquad
\kappa^-u_x(t,0^-)=\kappa^+u_x(t,0^+).
\]

\item[(iii)] On each half-line, the martingale component is given by
\[
q(t,x)=\sigma^\pm u_x(t,x),
\qquad x\in \mathbb R_\pm.
\]
\end{enumerate}
Moreover, $u$ is uniquely determined by the representation \eqref{representationsmoothcase} within this class of functions.
\end{theorem}
 
\begin{proof}Under $\mathbb Q$, the forward SDE reads
\[
dX_s
=
\mu(X_s)\,ds+\sigma(X_s)\,dW_s^{\mathbb Q},
\qquad X_t=x,
\]
where $
\mu(x):=-(b(x)+\sigma(x)\gamma).$
Since $b(\cdot)$ and $\sigma(\cdot)$ are piecewise constant, they are bounded and measurable. Moreover,
\[
\sigma^2(x)\ge \min\{\rho^-a^-,\rho^+a^+\}>0,
\qquad x\in\mathbb R,
\]
so the diffusion coefficient is uniformly elliptic. In particular, the one-dimensional SDE above admits a unique strong solution.

Concerning the Markov property. Let us fix $0\le t\le s\le T$ and define, for $r\ge0$,
\[
\widetilde W_r:=W_{t+r}^{\mathbb Q}-W_t^{\mathbb Q}.
\]
Then $(\widetilde W_r)_{r\ge0}$ is a Brownian motion independent of $\mathcal F_t$. Writing
\[
Y_r:=X_{t+r},\qquad r\ge0,
\]
so that $Y$ satisfies
\[
Y_r
=
X_t+\int_0^r \mu(Y_\ell)\,d\ell
+\int_0^r \sigma(Y_\ell)\,d\widetilde W_\ell.
\]
Hence, by pathwise uniqueness, $Y$ coincides with the unique strong solution of the same SDE started from $X_t$ at time $0$ and driven by the shifted Brownian motion $\widetilde W$.

Therefore, for every bounded Borel function $\varphi$,
\[
\E^{\Q}\!\left[\varphi(X_s)\mid\mathcal F_t\right]
=
\E^{\Q}\!\left[\varphi(Y_{s-t})\mid\mathcal F_t\right]
=
\P_{s-t}\varphi(X_t),
\]
where
\[
\P_r\varphi(x):=\E^{\Q}\!\left[\varphi(X_r^{\,x})\right],
\qquad r\ge0.
\]
Thus $X$ is a Markov process. Since the same argument applies at stopping times, $X$ is in fact a strong Markov process.

Consequently, the martingale problem associated with the generator
\[
\mathcal L u
=
\kappa(x) u_{xx}
-
(b(x)+\sigma(x)\gamma) u_x,\qquad x\neq 0
\]
is well posed. 

Fix $x\neq 0.$ On each half-line $\mathbb R_- = (-\infty,0)$ and $\mathbb R_+ = (0,\infty)$, the coefficients are constant, so that, from the smooth case,  Theorem \ref{thm:explicitBSPDE}  applies locally and yields the representation given in \eqref{representationsmoothcase}.
%\[
%u(t,x)=\mathbb E^{\mathbb Q}
%\Bigg[e^{c(T-t)} g(X_T)+
%\int_t^T e^{c(s-t)} f(s,X_s)\,ds
%\;\Bigg|\;\mathcal F_t
%\Bigg]\quad \text{for}\; x\neq0. \]

When $x=0,$ the interface behavior is determined by integrating the deterministic parabolic equation
\[
u_t
+
\kappa(x) u_{xx}
-
(b(x)+\sigma(x)\gamma) u_x
+
c u
+
f
=
0
\]
over a small interval $(-\varepsilon,\varepsilon)$ and let $\varepsilon\downarrow 0$. Since $\kappa$ has a jump at $0$, the integration by parts yields the flux condition
\[
\kappa^- u_x(t,0^-)=\kappa^+ u_x(t,0^+).
\]
Continuity of $u$ follows from parabolic regularity and the probabilistic representation, so that
\[
u(t,0^-)=u(t,0^+).
\]
Thus $u$ satisfies the standard transmission conditions.

Finally, the identification of the martingale term follows directly from Theorem \ref{thm:explicitBSPDE}. Indeed, on each half-line $\mathbb R_\pm$, the coefficients are constant, so the smooth-case result applies locally and yields
\[
q(t,x)=\sigma^\pm u_x(t,x),
\qquad x\in\mathbb R_\pm
\]
which proves $(iii)$.

Combining the above steps completes the proof.
\end{proof}

\section{Examples of Applications: Heat Regulation}
In this section, we illustrate the applicability of the theoretical results established in the previous sections through some concrete examples arising in heat regulation.
\subsection{Smooth Case: Nonzero-Sum Stochastic  with Constant Coefficients}
Heat regulation in conductive materials plays a central role in many
industrial applications, including thermal energy storage systems,
chemical reactors, electronic cooling, and advanced manufacturing processes.
Maintaining a stable and spatially uniform temperature distribution is
essential to prevent thermal stresses, structural damage, or loss of
performance.

\medskip 

The system consists of a homogeneous material occupying the one-dimensional
spatial domain $\mathbb R$. The temperature distribution is denoted by
$Y(t,x)$, where $t\ge 0$ is time and $x\in\mathbb R$ is the spatial variable.
The material is characterized by constant physical parameters $(\rho_0,a_0,b_0)$:
\begin{itemize}
\item density $\rho_0>0$,
\item thermal conductivity $a_0>0$,
\item transport  coefficient $b_0\in\mathbb R$.
\end{itemize}

Here, $\rho_0, a_0 > 0$ and $b_0 \in \mathbb{R}$ are given constants.

Two independent actuators are used to regulate the temperature.
Their spatial influence is described by functions
$\alpha_1,\alpha_2\in L^2(\mathbb R)$, while the scalar control processes
$u_1(t)$ and $u_2(t)$ represent the time dependent heating intensities.

%%%%%%%%%%%%%%Graph%%%%%%%%%%%%%%%%%%%%%%%%%%%%%%%%%
\begin{figure}[ht!]
\centering
\begin{tikzpicture}[thick,scale=2]

% ======= Main block (homogeneous medium) =======
\coordinate (A1) at (0,0);
\coordinate (A2) at (0,1);
\coordinate (A3) at (2.2,1);
\coordinate (A4) at (2.2,0);

% Back/top shift for 3D effect
\coordinate (B1) at (0.35,0.35);
\coordinate (B2) at (0.35,1.35);
\coordinate (B3) at (2.55,1.35);
\coordinate (B4) at (2.55,0.35);

% Front rectangle
\draw[very thick] (A1)--(A2)--(A3)--(A4)--cycle;

% 3D edges
\draw[dashed] (A1)--(B1);
\draw[very thick] (A2)--(B2);
\draw[very thick] (A3)--(B3);
\draw[very thick] (A4)--(B4);
\draw[dashed] (B1)--(B4);
\draw[very thick] (B2)--(B3);
\draw[very thick] (B3)--(B4);
\draw[dashed] (B1)--(B2);

% Fill the homogeneous medium (single phase)
\draw[fill=purple,opacity=0.45] (A1)--(A2)--(A3)--(A4)--cycle;
\draw[fill=purple,opacity=0.35] (A2)--(B2)--(B3)--(A3)--cycle;
\draw[fill=purple,opacity=0.35] (A4)--(B4)--(B3)--(A3)--cycle;

% ======= Two actuator regions (bands on the front face) =======
% Region 1 band
\coordinate (R11) at (0.25,0.15);
\coordinate (R12) at (0.25,0.85);
\coordinate (R13) at (0.95,0.85);
\coordinate (R14) at (0.95,0.15);

% Region 2 band
\coordinate (R21) at (1.25,0.15);
\coordinate (R22) at (1.25,0.85);
\coordinate (R23) at (1.95,0.85);
\coordinate (R24) at (1.95,0.15);

% ======= Axis =======
%\draw[ultra thick,->] (-0.25,-0.25) -- (2.65,-0.25)
%node[right] {$x:$ Spatial Domain};

% ======= Homogeneous parameter label =======
\node at (1.28,0.7) {\bf \small Constants coefficients};
\node at (1.1,0.49) {$(\rho_0,a_0,b_0)$};

% ======= Controls arrows (game) =======
\draw[cyan,ultra thick,<-] (1,1.3)-- (1.3,1.5) node[align=center] {\color{cyan} Controls Game $(u_1(t),u_2(t))$\\};
\draw[cyan,ultra thick,<-] (1.8,1.3)-- (1.3,1.5);

% ======= Noise / disturbance =======
\draw[gray,ultra thick,<-] (2.55,1.05)--(2.8,1.1)
node[right] {Noise};
\node[gray] at (3.15,0.95) {$\sigma_0\,dB(t)$};
\end{tikzpicture}
\caption{Homogeneous thermal medium with two distributed actuators}
\label{fig:homogeneous_heat}
\end{figure}
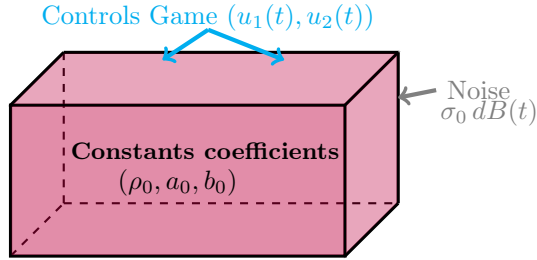

We consider the following stochastic evolution equation with constant coefficients:
\begin{equation}\label{forward-heat}
\left\{
\begin{aligned}
dY(t,x)
&=\Big[\kappa_0 Y_{xx}(t,x)
+b_0 Y_x(t,x)+\alpha_1(x)u_1(t)
+\alpha_2(x)u_2(t)
\Big]\,dt
+ \sigma_0\,dB(t),\\
Y(0,x)&=\xi(x),
\end{aligned}
\right.
\end{equation}
for $(t,x) \in [0,T] \times \mathbb{R}$, with
$
\kappa_0=\dfrac{\rho_0 a_0}{2}.
$

The second-order term represents diffusion, while the first-order term accounts for a constant drift. 
The noise term is driven by a standard Brownian motion $B(t)$, and $\sigma_0 > 0$ denotes the noise intensity.

\medskip 

The control enters additively through the space-dependent source terms 
$\alpha_i(x) u_i(t),i=1,2$, where $\alpha_i(x)$ are smooth weighting 
functions localizing the influence of each control.

\medskip 

Under the above assumptions, the stochastic evolution equation \eqref{forward-heat} admits a unique mild solution in $H=L^2(\mathbb R)$, according to Theorem~\ref{exitforward}.

\medskip 
Each player $i \in \{1,2\}$ seeks to minimize the quadratic cost functional
\begin{equation}\label{eq:focused_cost} J^i(u_1,u_2) = \mathbb{E} \left[\frac{\gamma_i}{2} \int_0^T  u^2_i(t)\,dt +  \frac{\gamma_3}{2}\int_{\mathbb{R}} Y^2(T,x)\,dx \right], \end{equation}
where $\gamma_i > 0,i=1,2,3$ penalizes the control effort. This defines a nonzero-sum stochastic  game.

\medskip 
Such a modeling framework is relevant to a broad class of thermal regulation
problems encountered in engineering and applied sciences, including:
\begin{itemize}
\item thermal energy storage and heat management systems,
\item materials with distributed or embedded heating mechanisms,
\item structural components operating under repeated thermal variations,
\item electronic and microelectronic devices subject to variable heat loads.
\end{itemize}

%Each actuator $u_i(t),i=1,2,$ seeks to regulate the temperature while minimizing its own
%energy expenditure. Since the control objectives may differ, the problem naturally leads to a nonzero-sum stochastic differential game formulation.

\medskip 

The Hamiltonian associated with player $i$ is defined by
\begin{equation}\label{hamiltonian}
\begin{aligned}
H^i(t,y,u_1,u_2,p^i,q^i)
&=
\langle p^i,\mathcal A y\rangle
+
u_1\langle p^i,\alpha_1\rangle
+
u_2\langle p^i,\alpha_2\rangle
+
\langle q^i,\sigma_0\rangle
+
\frac{\gamma_i}{2}u_i^2,
\end{aligned}
\end{equation}
where
\[
\mathcal A=\kappa_0\partial_{xx}+b_0\partial_x.
\]

According to Remark~\ref{Remark_adjoint}, the adjoint operator of $\mathcal A$ is given by
\[
\mathcal A^*=\kappa_0\partial_{xx}-b_0\partial_x.
\]

The adjoint processes $(p^i,q^i)$ satisfy the BSPDE
\begin{equation}\label{backward-heat}
\left\{
\begin{aligned}
-dp^i(t,x)
&=
\mathcal A^* p^i(t,x)
dt
-
q^i(t,x)\,dB(t),\\
p^i(T,x)&=Y(T,x).
\end{aligned}
\right.
\end{equation}

Observe that \eqref{backward-heat} is a linear BSPDE with constant coefficients. Therefore, it falls within the framework of Theorem~\ref{thm:explicitBSPDE}, and the representation formula \eqref{representationsmoothcase} applies the representation:
\begin{equation}
p^i(t,x)
=
\mathbb E
\left[
\phi\!\left(
x - b_0(T-t)
+
\sqrt{c}\,(B_T-B_t)
\right)
\Big|\mathcal F_t
\right].
\end{equation}

Equivalently,
\begin{equation}
p^i(t,x)
=
\int_{\mathbb R}
\phi(y)
\frac{1}{\sqrt{2\pi c (T-t)}}
\exp\!\left(
-\frac{(y-x+b_0(T-t))^2}{2c(T-t)}
\right)
dy.
\end{equation}

\medskip 

We now derive the optimal controls using the stochastic maximum principle.
For each player $i\in\{1,2\}$, the Hamiltonian $H^i$ is convex with respect to the control variable $u_i$. Therefore, a necessary and sufficient condition for optimality is given by the first-order condition
\[
\frac{\partial H^i}{\partial u_i}(t,Y,u_1,u_2,p^i,q^i)=0.
\]

From the expression of the Hamiltonian \eqref{hamiltonian}, we compute
\[
\frac{\partial H^i}{\partial u_i}
=
\langle p^i,\alpha_i\rangle
+
\gamma_i\,u_i.
\]

Setting this derivative equal to zero yields
$$
\langle p^i(t,\cdot),\alpha_i\rangle
+
\gamma_i\,u_i^*(t)
=
0,
$$
and hence the first-order optimality condition yields
\[
u_i^*(t)
=
-\frac{1}{\gamma_i}\,\langle p^i(t,\cdot),\alpha_i\rangle,
\qquad i=1,2.
\]

Substituting the explicit formula for $p^i$,
\begin{equation*}
\widehat u_i(t)
=
-\frac{1}{\gamma_i}
\int_{\mathbb R}
\phi(y)
\left[
\int_{\mathbb R}
\alpha_i(x)
\frac{1}{\sqrt{2\pi c (T-t)}}
\exp\!\left(
-\frac{(y-x+b_0(T-t))^2}{2c(T-t)}
\right)
dx
\right]
dy.
\end{equation*}

Thus, the optimal controls are given explicitly as Gaussian convolutions of
$\alpha_i$ and $\phi$.

Since each player minimizes its own cost functional and the optimality condition is satisfied simultaneously for $i=1,2$, the pair $(u_1^*,u_2^*)$ defines a Nash equilibrium.

\subsection{Non-Smooth Case: Nonzero-Sum Game with Piecewise Constant Coefficients}

We now extend the previous model to the case where the material exhibits
different physical properties on the two regions $\mathbb{R}_-$ and $\mathbb{R}_+$.
We assume that the material is characterized by piecewise constant coefficients
as in \eqref{coeff-piecewise}, where $a^\pm>0$, $\rho^\pm=1$, and $b^\pm\in\mathbb{R}$.

Define
\[
\kappa(x)=\frac{a(x)}{2},
\qquad
\sigma(x)=\sqrt{a(x)},
\qquad
b(x)=b^\pm \text{ on } \mathbb{R}_\pm.
\]

\medskip

The temperature field $Y(t,x)$ satisfies the stochastic heat equation
\begin{equation}\label{forward-piecewise-heat}
\left\{
\begin{aligned}
dY(t,x)
&=
\Big[
\kappa(x) Y_{xx}(t,x)
+
b(x) Y_x(t,x)
+
\alpha_1(x)u_1(t)
+
\alpha_2(x)u_2(t)
\Big]dt
+
\sigma(x)\,dB(t),\\
Y(0,x)&=\xi(x).
\end{aligned}
\right.
\end{equation}

Under the above assumptions, Theorem~\ref{exitforward} ensures the existence
and uniqueness of a mild solution.

\medskip

The Hamiltonian associated with player $i$ is defined by
\begin{equation}\label{hamiltonian2}
\begin{aligned}
H^i(t,y,u_1,u_2,p^i,q^i)
&=
\langle p^i,\mathcal A y\rangle
+
u_1\langle p^i,\alpha_1\rangle
+
u_2\langle p^i,\alpha_2\rangle
+
\langle q^i,\sigma(\cdot)\rangle
+
\frac{\gamma_i}{2}u_i^2,
\end{aligned}
\end{equation}
where
\[
\mathcal A
=
\kappa(x)\,\partial_{xx}
+
b(x)\,\partial_x.
\]

Each player $i \in \{1,2\}$ seeks to minimize the quadratic cost functional
\begin{equation}
J^i(u_1,u_2)
=
\mathbb{E}
\left[
\frac{\gamma_i}{2} \int_0^T u_i^2(t)\,dt
+
\frac{\gamma_3}{2}\int_{\mathbb{R}} Y^2(T,x)\,dx
\right],
\end{equation}
where $\gamma_i>0$, $i=1,2,3$.

\medskip

The adjoint processes $(p^i,q^i)$ satisfy the BSPDE
\begin{equation}\label{backward-piecewise-heat}
\left\{
\begin{aligned}
-dp^i(t,x)
&=
\mathcal A^* p^i(t,x)\,dt
-
q^i(t,x)\,dB(t),\\
p^i(T,x)&=Y(T,x),
\end{aligned}
\right.
\end{equation}
where $\mathcal A^*$ is given in Lemma~\ref{lem:adjoint-piecewise}.

\medskip

Equation \eqref{backward-piecewise-heat} is a BSPDE with piecewise constant
coefficients and therefore falls within the framework of
Theorem~\ref{thm:piecewise-BSPDE}. Applying the theorem with $u=p^i$, we obtain
\[
p^i(t,x)
=
\mathbb{E}^{\mathbb Q}
\Big[
e^{c(T-t)} Y\big(T,X_T^{t,x}\big)
\mid \mathcal F_t
\Big].
\]

Equivalently, the adjoint state admits the integral representation
\begin{equation}\label{piecewise-kernel-representation}
p^i(t,x)
=
e^{c(T-t)}
\int_{\mathbb R}
Y(T,y)\,
p(T-t,x,y)\,dy,
\end{equation}
where $p(\tau,x,y)$ denotes the transition density of the diffusion
\eqref{forward-piecewise}.

\medskip

The kernel $p(\tau,x,y)$ is the fundamental solution of the two-phase
parabolic operator
\[
\partial_t u
=
\kappa(x) u_{xx}
-
(b(x)+\sigma(x)\gamma)u_x,
\]
and satisfies the transmission conditions at $x=0$.

For $x,y$ lying in the same region, the kernel behaves locally like a Gaussian:
\[
p(\tau,x,y)
\approx
\frac{1}{\sqrt{2\pi \sigma_\pm^2 \tau}}
\exp\!\left(
-\frac{(y-x+\beta_\pm \tau)^2}{2\sigma_\pm^2 \tau}
\right),
\]
where
\[
\beta_\pm = b^\pm+\sigma^\pm\gamma.
\]

In general, $p(\tau,x,y)$ can be expressed as a superposition of Gaussian
terms corresponding to trajectories that may cross the interface $x=0$.

\medskip

Substituting \eqref{piecewise-kernel-representation} into the optimality condition
\[
\widehat u_i(t)
=
-\frac{1}{\gamma_i}
\int_{\mathbb R}\alpha_i(x)p^i(t,x)\,dx,
\]
and applying Fubini's theorem yields
\[
\widehat u_i(t)
=
-\frac{e^{c(T-t)}}{\gamma_i}
\int_{\mathbb R}
Y(T,y)
\left[
\int_{\mathbb R}
\alpha_i(x)\,p(T-t,x,y)\,dx
\right]
dy.
\]
The adjoint process satisfies the transmission conditions at the interface:
\[
p^i(t,0^-)=p^i(t,0^+),
\qquad
\kappa^-\, p_x^i(t,0^-)=\kappa^+\, p_x^i(t,0^+).
\]
As in the constant coefficient case, the optimal controls are given by
\[
u_i^*(t)
=
-\frac{1}{\gamma_i}\,\langle p^i(t,\cdot),\alpha_i\rangle,
\qquad i=1,2.
\]

\bigskip
\bigskip

{\bf Acknowledgments:} This work was supported by Supported by the Swedish Research Council grant (2020-04697) and the National Institute for Mathematical Sciences and Interactions (INSMI), CNRS, as part of the
PEPS JCJC 2025 call; and the project GAME (ANR-25-BFC 261359) of the French National Research Agency (ANR). 
%We thank the Editor, Associate Editor and anonymous reviewer for the valuable comments and remarks that have allowed to improve the quality of this work

\bibliography{Biblio}

\end{document}